\documentclass[12pt]{article}         

\usepackage{amsmath,mathrsfs,amsthm,amstext,amscd,url,amssymb,mathtools,faktor,dsfont}
\usepackage{amsfonts}
\usepackage{mathrsfs}
\usepackage{enumerate}
\usepackage{comment}
\usepackage{braket}
\usepackage[export]{adjustbox}
\usepackage{centernot}
\usepackage{wasysym}
\usepackage{acronym}
\usepackage{paralist}
\usepackage{xspace}
\usepackage{hyperref}
\usepackage{scrextend}
\usepackage{latexsym,textcomp,multirow}
\usepackage{graphicx}
\usepackage{textcomp}
\usepackage{empheq}
\usepackage{siunitx}
\usepackage{wrapfig}
\usepackage{chngpage}
\usepackage{graphics,subfigure}
\usepackage{tabularx}
\usepackage{booktabs}
\usepackage{epigraph}
\usepackage[inline]{enumitem}
\usepackage[title]{appendix}
\usepackage{tikz}
\usepackage{siunitx}
\usepackage{csquotes}
\allowdisplaybreaks

\usepackage{etoolbox}

\DeclarePairedDelimiterX{\normi}[1]
{|||}
{|||}
{\ifblank{#1}{\:\cdot\:}{#1}}

\numberwithin{equation}{section}

\newtheorem{theorem}{Theorem}[section]
\newtheorem{lemma}[theorem]{Lemma}
\newtheorem{assumption}{Assumption}[section]
\newtheorem{proposition}[theorem]{Proposition}
\newtheorem{corollary}[theorem]{Corollary}
\newtheorem{definition}[theorem]{Definition}
\newtheorem{remark}[theorem]{Remark}

\usepackage{geometry}
\def\N{{\mathbb N}}
\def\Z{{\mathbb Z}}

\def\R{{\mathbb R}}

\newcommand{\E}{{\mathbb E}}
\renewcommand{\P}{{\mathbb P}}

\newcommand{\SEPa}{\text{\normalfont SEP($\boldsymbol \alpha$)}}
\newcommand{\RWa}{\text{\normalfont{RW($\boldsymbol \alpha$)}}}
\newcommand{\RWo}{\text{\normalfont{RW($\boldsymbol \omega$)}}}
\newcommand{\Aa}{A^{\boldsymbol \alpha}}
\newcommand{\Ao}{A^{\boldsymbol \omega}}
\newcommand{\Sa}{S^{\boldsymbol \alpha}}

\newcommand{\Xa}{{{\mathcal X}^{\boldsymbol \alpha}}}
\newcommand{\Xaf}{{\mathcal X}_f^{\boldsymbol \alpha}}
\newcommand{\lx}{\lambda^{\boldsymbol \alpha}_x}
\newcommand{\lox}{\lambda^{\boldsymbol \omega}_x}
\newcommand{\Da}{D^{\boldsymbol \alpha}}
\newcommand{\La}{L^{\boldsymbol \alpha}}

\newcommand{\RCM}{\text{\normalfont RCM}}
\newcommand{\ind}{\mathbf{1}}

\newcommand{\abs}[1]{\ensuremath{\left\lvert #1 \right\rvert}} 
\newcommand{\norm}[1]{\ensuremath{\left\lVert #1 \right\rVert}} 

\newcommand{\dd}{\text{\normalfont d}}

\allowdisplaybreaks

\usepackage[inline]{enumitem}

\newcommand{\closure}[2][3]{%
	{}\mkern#1mu\overline{\mkern-#1mu#2}}

\usepackage{sectsty}

\sectionfont{\fontsize{13}{13}\selectfont}
\subsectionfont{\fontsize{11}{11}\selectfont}
\subsubsectionfont{\fontsize{10}{10}\selectfont}

\newcommand{\footremember}[2]{%
	\footnote{#2}
	\newcounter{#1}
	\setcounter{#1}{\value{footnote}}%
}

\makeatletter
\let\@fnsymbol\@arabic
\makeatother

\title{{\bf \Large Hydrodynamics for the partial   exclusion process\\ in random environment}}

\author{Simone Floreani\footremember{Delft}{Delft Institute of Applied Mathematics, TU Delft,
		Delft, The Netherlands. \href{mailto:s.floreani@tudelft.nl}{s.floreani@tudelft.nl}.	}, 
	Frank Redig\footremember{Delft1}{Delft Institute of Applied Mathematics, TU Delft,
		Delft, The Netherlands. \href{mailto:f.h.j.redig@tudelft.nl}{f.h.j.redig@tudelft.nl}.	},
	Federico Sau\footremember{Paris5}{Institute of Science and Technology, IST Austria, Klosterneuburg, Austria. \href{mailto:federico.sau@ist.ac.at}{federico.sau@ist.ac.at}. 
} }
\providecommand{\keywords}[1]
{
	\small	
	\textbf{Keywords:}\quad #1
}
\begin{document}
\maketitle

	\begin{abstract}
		In this paper, we introduce a  random environment for the exclusion process in $\Z^d$ obtained by assigning  a maximal occupancy to each site. This maximal occupancy is allowed to randomly vary among sites, and partial exclusion occurs. 
		Under the assumption of ergodicity under translation and uniform ellipticity of the environment, we derive  a quenched hydrodynamic limit in path space by strengthening the mild solution approach initiated in \cite{nagy_symmetric_2002} and \cite{faggionato_bulk_2007}. To this purpose, we prove, employing the technology developed for the random conductance model, a homogenization result in the form of an arbitrary starting point quenched invariance principle for a single particle in the same environment, which is a result of independent interest. The self-duality property of the partial exclusion process allows us to transfer this homogenization result to the particle system and, then,  apply the tightness criterion in \cite{redig_symmetric_2020}.	
	\end{abstract}


\section{Introduction}\label{section introduction}

In recent years there has been extensive study of the scaling limit of random walks in both static and dynamic random environment. In this realm, the \textit{random conductance model} (RCM) takes a  prominent place. Various analytic tools have been developed to prove scaling properties such as quenched invariance principles, local central limit theorems as well as detailed estimates on the random walks such as heat kernel bounds (see, e.g., \cite{biskup_recent_2011} for an overview on the subject).

A natural next step
is to consider interacting particle systems in random environment, where particles  model transport of mass or energy, while the random environments model, e.g.,  impurities or defects in  the conducting material. The macroscopic effects of the environment may be studied  through  scaling limits such as hydrodynamic limits, fluctuations and large deviations around the hydrodynamic limit, as well as via the study of non-equilibrium behavior of systems coupled to reservoirs which, in random environment, is still a challenge. 

Due to the presence of the random environment, these systems are typically \textit{non-gradient} and standard gradients  methods to study the hydrodynamic behavior  do not carry on. Nevertheless, interacting particle systems with (self-)duality are especially suitable to make the step from single-particle scaling limits towards the derivation of the macroscopic equation for the many-particle system. Indeed, in such systems, the macroscopic equation can be guessed from the behavior of the expectation of the local particle density which, in turn, amounts to understand the scaling behavior of a single \textquotedblleft dual\textquotedblright\ particle. However, this intuitive \textquotedblleft transference principle\textquotedblright\  from the scaling limit of one random walker to the macroscopic equation has to be made rigorous.

\subsection{Model}\label{section:model}

In the present work, we introduce a  random environment for the exclusion process in $\Z^d$ obtained by assigning  a maximal occupancy  $\alpha_x\in\N$ to each site $x\in\Z^d$ and we study its hydrodynamic limit.

\begin{figure}[h]
	\centering
	\includegraphics[width=0.9\linewidth]{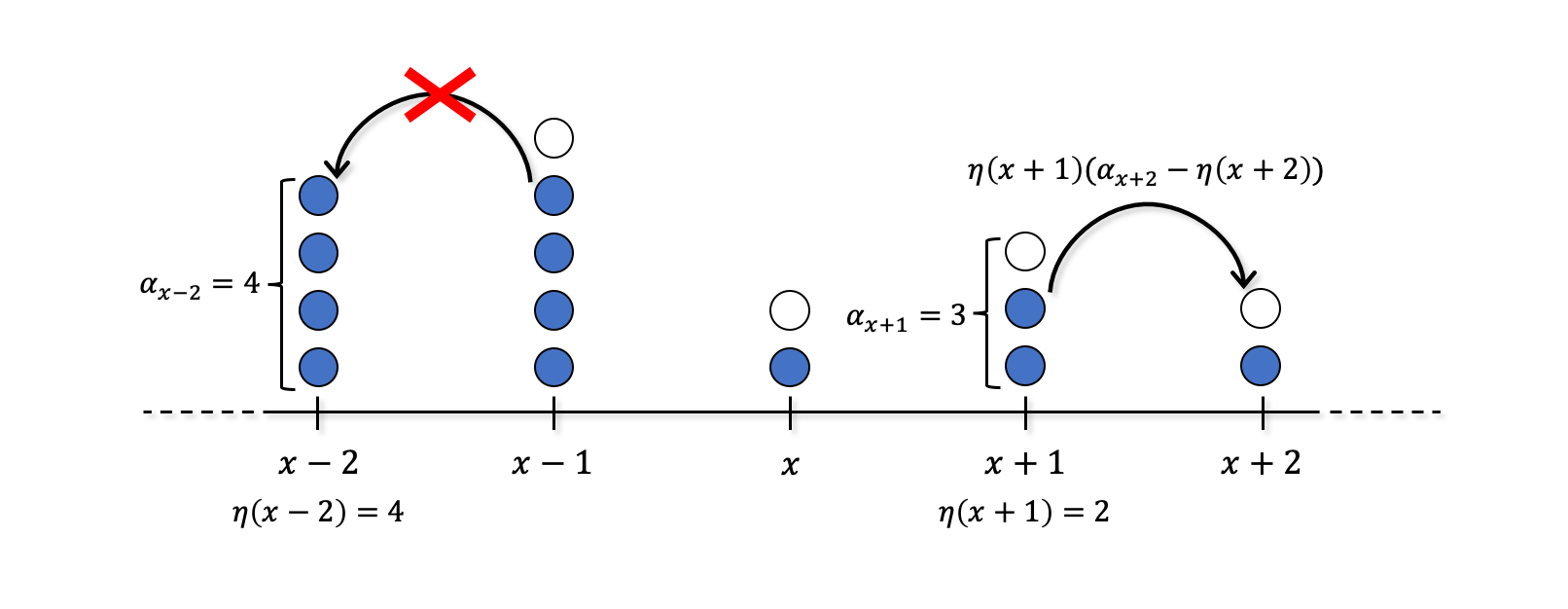} 
	\caption{Schematic description of the one-dimensional  partial exclusion process in the  environment $\boldsymbol \alpha = \{\alpha_x,\, x \in \Z\}$, where $\alpha_x \in \N$ denotes the maximal occupancy of site $x \in \Z$.}
	\label{fig:hdl}
\end{figure}
In what follows,  we refer to \textit{random environment} as the collection $\boldsymbol \alpha=\{\alpha_x,\, x\in\Z^d\}$, for which 
we assume the following.

\begin{assumption}[{Ergodicity and uniform ellipticity of $\boldsymbol{\alpha}$}]\label{definition  disorder} We fix a constant $\mathfrak c\in \N$ for which  the random environment $\boldsymbol \alpha=\{\alpha_x,\, x\in\Z^d\}$  is chosen according to a distribution $\mathcal P$ on $\{1,...,\mathfrak c\}^{\Z^d}$, which is stationary and ergodic under translations $\{\tau_x,\, x \in \Z^d\}$ in $\Z^d$. 
\end{assumption}
In particular, all realizations  $\boldsymbol \alpha$ of the random environment satisfy the following uniform upper and lower bounds: 	
\begin{equation}\label{equation uniform ellipticity}
1 \leq \alpha_x \leq \mathfrak c\ ,\qquad x \in \Z^d\ .
\end{equation}

Let us introduce the exclusion process in the environment $\boldsymbol \alpha$ (see Figure \ref{fig:hdl}) and indicate  the configuration of particles by $\eta= \{\eta(x),\, x\in\Z^d\}$,  consisting of a collection of occupation variables  indexed by the sites of $\Z^d$. These variables indicate the number of particles at each site, i.e.,
\begin{equation*}
\eta(x)\coloneqq	 \text{number of particles at}\ x\ .
\end{equation*}
We define the configuration space  $\mathcal X^{\boldsymbol \alpha}$ (endowed with the product topology) as
\begin{equation}
\mathcal X^{\boldsymbol \alpha}\coloneqq\Pi_{x\in \Z^d}\{0,...,\alpha_x\}\ ;
\end{equation}
here  the superscript emphasizes  the dependence of the configuration space  on the realization of the environment. Hence, given a realization $\boldsymbol \alpha$ of the random environment, the \textit{partial (simple) exclusion process in the  environment $\boldsymbol \alpha$}, abbreviated by $\SEPa$, is the Markov process on $\mathcal X^{\boldsymbol \alpha}$ whose generator acts on bounded cylindrical  functions $\varphi:\mathcal X^{\boldsymbol \alpha}\to \R$, i.e., functions which  depend only on a finite number of occupation variables,  as follows (all throughout the paper,  $\left|\cdot\right|$ will always denote the Euclidean norm):

\begin{align}\label{equation generator inhomog SEP}
\La \varphi(\eta)=\sum_{\substack{\{x,y\}\subseteq \Z^d,\\|x-y|=1}}&\left\{\begin{array}{r}\eta(x)(\alpha_y-\eta(y))\left(\varphi(\eta^{x,y})-\varphi(\eta)\right)\\[.15cm]	+\,\eta(y)(\alpha_x-\eta(x))\left(\varphi(\eta^{y,x})-\varphi(\eta)\right)\end{array}\right\}\ .
\end{align}
In the above formula, $\eta^{x,y}$ denotes the configuration obtained from $\eta$ by removing a particle (if any) from the  site $x$  and adding a particle to the site $y$, 	 i.e.,
\begin{equation}\label{equation configuration after the jump}
\eta^{x,y}=\begin{dcases}
\eta -\delta_x +\delta_y\ \ &\text{if}\ \eta(x)\ge 1\ \text{and}\ \eta(y)< \alpha_y\\
\eta\ &\text{otherwise}\ .
\end{dcases}
\end{equation}

Condition \eqref{equation uniform ellipticity}  ensures the existence of the process (see, e.g., \cite[Chapter 1]{liggett_interacting_2005-1}), which we call $\{\eta_t,\, t \geq 0 \}$, defined via the generator \eqref{equation generator inhomog SEP}. We highlight that  $\SEPa$ 	  is a inhomogeneous variant of the partial exclusion process considered in \cite{schutz_non-abelian_1994} (see also \cite{giardina_duality_2009}), where $\alpha_x=m$ for any $x \in \Z^d$ and $m$ is a  natural number, while, for the choice $\alpha_x=1$ for any $x\in\Z^d$, we recover the  simple symmetric exclusion process in $\Z^d$ (see, e.g., \cite{liggett_interacting_2005-1}). Moreover, if there is only one particle in the system, no interaction takes place and we are left with a single random walk in the environment $\boldsymbol{\alpha}$, that we call \emph{random walk in the random environment  $\boldsymbol \alpha$}, abbreviated by  $\RWa$. More precisely,	 $\RWa$ is the Markov process $\{X_t^{\boldsymbol \alpha},\, t \geq 0\}$ on $\Z^d$ with law $P^{\boldsymbol \alpha}$ induced by the	infinitesimal generator  given by
\begin{equation}\label{equation generator random walk in disorder}
\Aa f(x)\coloneqq\sum_{\substack{y\in \Z^d\\|y-x|=1}}\alpha_y\left(f(y)-f(x)\right)\ ,
\end{equation}
where $f:\Z^d\to \R$ is a bounded  function. For all $x \in \Z^d$,  let $X^{\boldsymbol \alpha,x} = \{X^{\boldsymbol \alpha,x}_t,\, t \geq 0 \}$ denote  the random walk $\RWa$ started in $x \in \Z^d$.

\medskip
\subsection{Quenched hydrodynamics and discussion of related literature}

The main result of this paper, Theorem \ref{theorem HDL}, states that,  under Assumption \ref{definition  disorder}, for almost every realization of the environment $\boldsymbol{\alpha}$,  the  path-space hydrodynamic limit of $\SEPa$ is a deterministic diffusion equation with a non-degenerate diffusion matrix not depending on the realization of the environment. To this purpose,  we run through the following steps. First, we show that  $\SEPa$ is dual to $\RWa$ and we express the occupation variables of $\SEPa$ at time $t$  as mild solutions of a lattice stochastic partial differential equation, linear in the drift. Then we  show that the microscopic disorder $\boldsymbol \alpha$ undergoes a homogenization effect, in the form of a quenched invariance principle for the random walks $\RWa$. In conclusion, we transfer this homogenization effect from the random walk to the interacting particle system via the aforementioned duality. 
To the essence, this transference principle boils down to the following two requirements:

\medskip

\begin{enumerate}[label={\normalfont(\roman*)},ref={\normalfont (\roman*)}]
	\item \label{it:transference1} \emph{Consistency of the initial conditions} (see Definition \ref{definition consistency } below) stating, roughly speaking, that a law of large numbers holds for the initial particle densities;
	
	\medskip
	
	\item \label{item: ASPQIP} The validity of a quenched homogenization result for the random walks $\RWa$ in the form of an \emph{arbitrary starting point quenched invariance principle} (see \eqref{eq:uniform_convergence_semigroups} below).
\end{enumerate}

\medskip

The mild solution approach to hydrodynamic limits in random environment has been initiated in \cite{nagy_symmetric_2002} in $\Z^d$ with  $d=1$  and further developed to any dimension and with less restrictive conditions in \cite{faggionato_bulk_2007}. Hence, the idea of deriving the hydrodynamic limit in random environment from a homogenization result for the dual random walk is not new. These works, though,  lack of a proof of path space tightness for the empirical density fields of the particle system, as more classical tightness criteria such as Aldous-Rebolledo and Censov (see, respectively, e.g., \cite{kipnis_scaling_1999} and \cite{de_masi_mathematical_1991}) do not apply when employing a mild solution representation for the density fields. 

On the other hand, along with the derivation of the limiting hydrodynamic equation, the proof of tightness for particle systems in  random environment has  been obtained  in several works by introducing the so-called \emph{corrected empirical density field}, an auxiliary process for which the evolution equation \textquotedblleft closes\textquotedblright\ and the aforementioned tightness criteria apply.  Thus, one has to face the extra step consisting in proving that the empirical density field and the corrected one are close in a suitable sense. The idea of the corrected empirical density field has been introduced in \cite{jara_quenched_2008} for the exclusion process with random conductances on $\Z^d$ with $d=1$ and later extended to the $d$-dimensional torus in  \cite{goncalves_scaling_2008}, with $d \geq 1$, and more general geometries in \cite{jara_hydrodynamic_2011}.
The construction of the corrected empirical density field as in \cite{goncalves_scaling_2008} is general enough to apply, by employing  the convergence of either the  random walk generators or the associated Dirichlet forms, also to different contexts, like in \cite{faggionato_hydrodynamic_2009} for a one-dimensional subdiffusive exclusion process, \cite{faggionato_zerorange_2010} for a zero range process with random conductances  and our context of site-varying maximal occupancy exclusion process. 
However, we believe that a general strategy to establish tightness and the hydrodynamic limit for sequences of tempered distribution-valued mild solutions may be of help when stochastic convolutions, although not being martingales, ensure a stronger space-time regularity of the stochastic processes as in the context, e.g., of Gaussian SPDEs. In \cite{redig_symmetric_2020}, in which the hydrodynamic limit of the simple exclusion process in presence of dynamic random conductances is studied,  a criterion for relative compactness, based on the notion of uniform stochastic continuity, has been presented. We apply this criterion to our context of partial exclusion, which has the advantages to directly apply to the sequence  of mild solutions and avoid the introduction of the auxiliary sequence of the corrected empirical density fields.

Next to the problem of ensuring relative compactness for the empirical density, another main challenge in the study of scaling limits of  particle systems in random environment is to  prove a homogenization result for the underlying environment. To get the desired homogenization result we employ, via a suitable random time change,  several concepts and results developed in the context of the random conductance model (RCM) (see, e.g., \cite{biskup_recent_2011}).
So far, the technology developed in the last two decades for RCM has not been employed in the context of particle systems in random environment, other types of convergence being preferred. In particular,  either $\Gamma$-convergence (see, e.g., \cite{jara_hydrodynamic_2011}) or  two-scale convergence (see, e.g., \cite{faggionato_cluster_2008, faggionato_stochastic_2020}) were employed to recover  quenched hydrodynamic limits for the simple exclusion process in more general settings than RCM with uniformly elliptic conductances.
 
For the $\RWa$ under Assumption \ref{definition  disorder}, one does not need such a level of generality and it is natural to try to use the existing quenched invariance principles for the random conductance model. However all quenched invariance principles for RCM (see, e.g., \cite{andres_invariance2013, andres_quenched_2018, andres_invariance2015, bella_quenched_2019, SidoraviciusSznitman}) are derived for the walk starting at the origin,  which is, in general, too weak as a convergence to ensure the quenched hydrodynamic limit for  the particle system. To fill the gap between quenched invariance principle and quenched hydrodynamic limit,  a  homogenization result involving the random walks $\RWa$ starting from \emph{all} spatial locations suffices. To this purpose, we choose to extend the quenched invariance principle valid for the random walk starting from the origin to walks starting from arbitrary sequences of starting points; we believe the latter to be a result of interest in its own right. Note that this strengthening is not trivial due to the lack of translation invariance of the law of the random walk in quenched random environment.

 The problem of deriving quenched arbitrary starting point invariance principles has been posed in \cite{Rodhes}  and only recently solved in 
\cite{chen_quenched_2015} for the  static random conductance model on the supercritical percolation cluster. 
In our context of random environment $\boldsymbol \alpha$, in order to  prove the quenched invariance principle with arbitrary starting positions for the dual random walk, we use  the formalism and ideas from \cite{chen_quenched_2015}. 

The connection between the quenched invariance principle in RCM and hydrodynamics in  random environment  seems to be promising, at least in the case of particle systems with self-duality, and this gives hope, for future works, to obtain path-space hydrodynamic limit also in degenerate environments. 
In conclusion, we remark that other strategies than self-duality to prove hydrodynamic limits for interacting particle systems in random environment are available and rely on the non-gradient methods (see, e.g., \cite{martinelli_hydrodynamic_2003}) and methods  based on  Riemann-characteristics for hyperbolic concentration laws (see, e.g., \cite{bahadoran2014}).

\

The remaining of the paper is organized as follows. In Section \ref{Section HDL} we state the main theorem -- the quenched hydrodynamic limit in path space -- and  explain the strategy of the proof in more detail. Section \ref{section invariance principles} is devoted to the arbitrary starting point quenched invariance principle and Section \ref{section proof HDL} to the proof of the hydrodynamic limit. The proofs of some  auxiliary results stated in the body of the paper are collected in separate  appendices at the end of the manuscript.

\section{Main result and strategy of the proof} \label{Section HDL}
As observable of the macroscopic
behavior of the interacting particle system, we consider  the empirical density fields, indicated, for all $N \in \N$, by  $\mathsf X^N=\{\mathsf X^N_t,\, t \geq 0 \}$.  Given, for a fixed realization of the environment $\boldsymbol \alpha$, a sequence of probability measures $\nu^{\boldsymbol \alpha}=\{\nu_N^{\boldsymbol \alpha}\}_{N\in\N}$ on the configuration space $\Xa$, for all $N \in \N$, the empirical density field $\mathsf X^N$ is a measure-valued process obtained as a function of  the system $\eta=\{\eta_t,\, t \geq 0\}$  as follows:
\begin{equation}\label{equation empirical density field}
\mathsf X^N_t \coloneqq\frac{1}{N^d}\sum_{x \in \Z^d} \delta_{\frac{x}{N}}\,  \eta_{tN^2}(x)\  ,
\end{equation}
where  $\eta$ is the process $\SEPa$ introduced in Section \ref{section introduction} initially distributed as $\nu_N^{\boldsymbol \alpha}$. We  refer to $\P^{\boldsymbol \alpha}_{\nu_N^{\boldsymbol \alpha}}$ as the probability measure on the Skorokhod space $\mathcal D([0,\infty),\Xa)$ of such process and let $\E^{\boldsymbol \alpha}_{\nu_N^{\boldsymbol \alpha}}$ denote the corresponding expectation, while  $\P^{\boldsymbol \alpha}_\eta$  and   $\E^{\boldsymbol \alpha}_\eta$ indicate the law and the corresponding expectation, respectively, of the process starting from the configuration $\eta$.	
We  note that the definition \eqref{equation empirical density field}   encodes a space-time diffusive rescaling of the microscopic system. Moreover, due to  the uniform upper bound in  \eqref{equation uniform ellipticity} on the maximal occupancies, 	 we  view (as done, e.g., in the textbook \cite[Chapter 2]{de_masi_mathematical_1991}) the
empirical density fields as  processes in $\mathcal D([0,\infty),\mathscr S'(\R^d))$; here, $\mathscr S'(\R^d)$ denotes the topological dual of the Schwartz class of smooth and rapidly decreasing functions $\mathscr S(\R^d)$ and $\mathcal D([0,\infty),\mathscr S'(\R^d))$  the Skorokhod space of $\mathscr S'(\R^d)$-valued c\`{a}dl\`{a}g trajectories. For further details on the construction and topologies of these spaces, we refer to, e.g.,  \cite[Chapter 2, Section 6]{de_masi_mathematical_1991},  \cite{mitoma_tightness_1983}, as well as \cite[Chapter 2, Section 4]{kallianpur_xiong_1995}. 	   \color{black}
Hence, for all $t \geq 0$, the action of $\mathsf X^N_t$ on the test function $G \in \mathscr S(\R^d)$ is given by
\begin{equation}\label{eq:density_fieldsG}
\mathsf X^N_t(G)\coloneqq \frac{1}{N^d}\sum_{x \in \Z^d} G(\tfrac{x}{N})\, \eta_{tN^2}(x)\ .
\end{equation}
Let us remark that this choice of the functional spaces $\mathcal D([0,\infty),\mathscr S'(\R^d))$, while being standard when studying fluctuation fields, is less canonical  in the context  of hydrodynamic limits (cf., e.g., \cite{kipnis_scaling_1999}). The motivation behind this choice is twofold. On the one side,  the nuclear structure of the pair $\mathscr S(\R^d)$ and $\mathscr S'(\R^d)$ allows, in Section \ref{section tightness}, to   employ Mitoma’s tightness criterion for processes in $\mathcal D([0,\infty),\mathscr S'(\R^d))$,  see \cite{mitoma_tightness_1983}.  On the other side, in Section \ref{section finite dimensional distributions}, we need that  $\mathscr S(\R^d)$ is dense and invariant under the action of the  semigroup on  $\mathcal C_0(\R^d)$ -- the Banach space of continuous and vanishing at infinity functions endowed with the supremum norm -- of the $d$-dimensional Brownian motion $\{B^\varSigma_t,\, t \geq 0 \}$ with diffusion matrix $\varSigma \in \R^{d \times d}$, i.e., the   strongly continuous and contraction semigroup $\{\mathcal S^\varSigma_t,\, t \geq 0 \}$ on $\mathcal C_0(\R^d)$ associated to the following second-order differential operator
\begin{equation*}
\mathcal A^\varSigma=\tfrac{1}{2}\nabla \cdot (\varSigma\, \nabla)\ .
\end{equation*}

As our goal is to study the limit of the $N$-th empirical density field $\mathsf X^N$ as $N$ goes to infinity,
we need to require that the initial particle configurations suitably rescale to a macroscopic  profile. We make this requirement precise in the following definition, in which $\mathscr P(\mathcal X^{\boldsymbol \alpha})$ denotes the space of probability measures on $\mathcal X^{\boldsymbol \alpha}$.

\begin{definition}[{Consistency of the initial conditions}]\label{definition consistency }
	We  say that, for a given environment $\boldsymbol \alpha$, a sequence of  probabilities $\nu^{\boldsymbol \alpha}\coloneqq\{\nu_N^{\boldsymbol \alpha}\}_{N\in\N}$ in $\mathscr P(\mathcal X^{\boldsymbol \alpha})$  is \emph{consistent to a continuous macroscopic  profile $ \bar{\rho} : \R^d \to [0,1]$}  if the following convergence
	\begin{equation}\label{equation consistency}
	\nu_N^{\boldsymbol \alpha}\left(  \left\{ \eta \in \mathcal X^{\boldsymbol \alpha}:\  \left| \frac{1}{N^d} \sum_{x \in \Z^d} G(\tfrac{x}{N})\eta(x) -   \int_{\R^d} G(u)\,\E_{\mathcal P}\left[\alpha_0\right]\bar{\rho}(u)\,	\dd u \right| > \delta	 \right\}\right)\ \underset{N\to \infty}\longrightarrow\ 0
	\end{equation}
holds  for all $G \in \mathscr S(\R^d)$ and $\delta > 0$.
\end{definition}

We are ready to state our main theorem,	 whose proof is deferred to Section \ref{section proof HDL} below. 
\begin{theorem}[{Hydrodynamic limit in quenched random environment}]\label{theorem HDL}
Let $\bar \rho : \R^d \to [0,1]$ 
	 be a continuous macroscopic  profile and, for all realizations of the environment $\boldsymbol \alpha$, let  $\nu^{\boldsymbol \alpha}=\{\nu_N^{\boldsymbol \alpha}\}_{N\in \N}$ be a  sequence of probabilities on $\mathscr P(\Xa)$. Recall Definition \ref{definition consistency }, define
	 \begin{equation}\label{eq:C}
	 	\mathfrak C\coloneqq \left\{\boldsymbol \alpha \in \{1,\ldots,\mathfrak c\}^{\Z^d}: \nu^{\boldsymbol \alpha}\ \text{is consistent with}\ \bar \rho \right\}\ ,
	 \end{equation}
 and assume that  $\mathcal P(\mathfrak C)=1$.

	Then, there exists two measurable subsets $\mathfrak A$ and $\mathfrak B \subseteq \{1,\ldots, \mathfrak c\}^{\Z^d}$ with $\mathcal P(\mathfrak A)=\mathcal P(\mathfrak B)=1$ (given, respectively, in \eqref{eq:A} and \eqref{eq:B} below) such that, for all $\boldsymbol \alpha \in \mathfrak A \cap \mathfrak B \cap \mathfrak C$ and for all $T > 0$, we have the following weak convergence in $\mathcal D([0,T],\mathscr S'(\R^d))$:
	\begin{equation}
	\left\{\mathsf X^N_t,\, t \in [0,T] \right\} \underset{N \to \infty}\Longrightarrow \left\{\pi^\varSigma_t,\, t \in [0,T] \right\}\ ,
	\end{equation}
	where the empirical density fields $\{\mathsf X^N_t,\, t \in [0,T]\}_{N \in \N}$ are  given as in \eqref{equation empirical density field} and
	\begin{equation}\label{eq:pi_t}
	\pi^\varSigma_t(\dd u)\coloneqq \E_{\mathcal P}\left[\alpha_0\right] \rho_t^\varSigma(u)\, \dd u\ ,
	\end{equation} with $\{\rho^\varSigma_t,\, t \geq 0 \}$ being the unique strong solution in $\R^d$  to 
	\begin{equation} \label{heat equation}
	\left\{
	\begin{array}{rcl}
	\partial_t \rho &=& \frac{1}{2} \nabla \cdot (\varSigma\,  \nabla \rho)\\[.1cm]
	\rho_0 &=& \bar{\rho} \ .
	\end{array}
	\right.
	\end{equation}
	In particular, the diffusion matrix $\varSigma \in \R^{d\times d}$ in \eqref{heat equation} and given in Proposition \ref{proposition:QFCLT_site_disorder} below is non-degenerate, symmetric, 	 positive-definite and does not depend on the particular realization of the environment.

\end{theorem}

\begin{remark}[{Existence and uniqueness of the limit}]\label{remark:uniquness}
	Let $\mathcal C_b(\R^d)$ denote the Banach space of continuous and bounded  functions from $\R^d$ to $\R$ endowed with the supremum norm.
	It is well-known (see, e.g., \cite[Chapter 2, Section 3.1, Theorem 1]{Evans}) that,  $\bar \rho$ being bounded and continuous, the   strong solution $\{\rho^\varSigma_t,\, t \geq 0\}$ to \eqref{heat equation} exists, is  unique and admits the following stochastic  representation in terms of the contraction and strongly continuous semigroup of Brownian motion $\{B^\varSigma_t,\, t \geq 0\}$ on $\mathcal C_b(\R^d)$,  still referred to -- with a slight abuse of notation --   as  $\{\mathcal S^\varSigma_t,\, t \geq 0\}$:
	\begin{equation}\label{eq:rho_t}
		\rho^\varSigma_t= \mathcal S^\varSigma_t\bar \rho\ ,\qquad t \geq 0\ .
	\end{equation}
Moreover, by \cite[Theorem 1.4]{holley_generalized_1978}, there exists a unique element $\left\{\pi_t,\, t \geq 0\right\}$ in the space of $\mathscr S'(\R^d)$-valued continuous trajectories $\mathcal C([0,\infty),\mathscr S'(\R^d))$ (see, e.g., \cite{holley_generalized_1978}, \cite[Chapter 2, Section 4]{kallianpur_xiong_1995}) such that either one of the following two identities hold  for all $t \geq 0$ and $G \in \mathscr S(\R^d)$:
\begin{equation}
	\pi_t(G)= \pi^{\bar \rho}(G)+\int_0^t \pi_s(\mathcal A^\varSigma G)\, \dd s
\qquad \text{or}\qquad
	\pi_t(G) = \pi^{\bar \rho}(\mathcal S^\varSigma_t G)\ ,
\end{equation}
where
\begin{equation}\label{eq:pi_0}
	\pi^{\bar \rho}(\dd u)\coloneqq \E_{\mathcal P}\left[\alpha_0\right]\bar \rho(u)\, \dd u\ .
\end{equation}
As a consequence of \eqref{eq:rho_t} and
\begin{equation}
	\int_{\R^d} \mathcal S^\varSigma_t G(u)\, H(u)\, \dd u = \int_{\R^d} G(u)\, \mathcal S^\varSigma_t H(u)\, \dd u\ ,\qquad G \in \mathscr S(\R^d)\ ,\ H \in \mathcal C_b(\R^d)\ ,\ t \geq 0\ ,	
\end{equation}	 such a unique element must coincide with $\{\pi^\varSigma_t,\, t \geq 0\}$ in \eqref{eq:pi_t}.	
\end{remark}

Before discussing the strategy of proof of Theorem \ref{theorem HDL}, we present  an ergodic theorem (Lemma \ref{lemma:ergodic_theorem} below) of importance at various stages of the paper; in particular,  this allows us to exhibit in Proposition \ref{proposition:slowly_varying} below a class of initial distributions  for $\SEPa$	 which verify the assumption of Theorem \ref{theorem HDL}.
 Preliminarily, we need the following definition.
\begin{definition}\label{definition:subsetF}
	A subset $\mathcal F$ of $\mathcal C_0(\R^d)$ is said to be \emph{equicontinuous} if
	\begin{equation}\label{eq:equicontinuous}
		\lim_{\delta\downarrow 0}	\sup_{\substack{u,v \in \R^d\\|u-v|<\delta}}\sup_{F \in \mathcal F}\left|F(u)-F(v) \right|=0
	\end{equation}
	holds,
	\emph{bounded} if
	\begin{equation}\label{eq:bounded}
		\sup_{F\in \mathcal F} \sup_{u\in \R^d}\left|F(u)\right|<\infty
	\end{equation}
	holds, and 
	\emph{uniformly integrable} if
	\begin{equation}\label{eq:uniform_integrability}
		\sup_{F \in \mathcal F}\left|F(u) \right|\leq f(u)\ ,\qquad u \in \R^d	\ ,
	\end{equation}
	holds for some function $f \in L^1(\R^d)\cap \mathcal C_0(\R^d)$.
\end{definition}

\begin{lemma}\label{lemma:ergodic_theorem}
	Under Assumption \ref{definition  disorder} on the environment, for $\mathcal P$-a.e.\ realization of the environment $\boldsymbol \alpha$, the following holds:
	\begin{quote}
		
		For all equicontinuous, bounded and uniformly integrable subsets $\mathcal F$ of $\mathcal C_0(\R^d)$ (see Definition \ref{definition:subsetF}), we have
		\begin{equation}
			\sup_{F \in \mathcal F}\left|\frac{1}{N^d}\sum_{x \in \Z^d} F(\tfrac{x}{N})\, \alpha_x - \E_{\mathcal P}\left[\alpha_0 \right]\int_{\R^d} F(u)\, \dd u	 \right|\underset{N\to \infty}\longrightarrow 0\ .
		\end{equation}
	\end{quote}
\end{lemma}
The proof of Lemma \ref{lemma:ergodic_theorem} can be found in Appendix \ref{appendix:ergodic} below.
 Moreover, we find convenient to define
\begin{equation}\label{eq:A}
	\mathfrak A\coloneqq \left\{\boldsymbol \alpha \in \{1,\ldots, \mathfrak c\}^{\Z^d}: \text{the claim in Lemma \ref{lemma:ergodic_theorem} holds for $\boldsymbol \alpha$}\right\}\ .
\end{equation}

\

By a detailed balance computation, it is simple to check  that the following  product measures 
\begin{equation}\label{equation reversible measures}
	\nu_p^{\boldsymbol \alpha}=\underset{x\in \Z^d} \otimes \text{Binomial} (\alpha_x, p)\ ,		
\end{equation}
are reversible measures for $\SEPa$, for all parameters $p \in [0,1]$. In general, if the parameter $p$ depends on the site $x \in \Z^d$, the corresponding Bernoulli product measures are not invariant for the exclusion dynamics. Nevertheless,  as shown in Proposition \ref{proposition:slowly_varying} below, such probability measures with  slowly varying parameter satisfy the assumptions of Theorem \ref{theorem HDL}. 
\begin{proposition}\label{proposition:slowly_varying}
	For all  $\boldsymbol \alpha \in \mathfrak A$ (see \eqref{eq:A}) and for all continuous profiles $\bar \rho:\R^d\to [0,1]$, the sequence of probabilities $\{\nu^{\boldsymbol \alpha,\bar \rho}_N\}_{N\in \N}$ in $\mathscr P(\Xa)$ given, for all $N\in \N$, by
	\begin{equation}
		\nu^{\boldsymbol\alpha, \bar \rho}_N\coloneqq \underset{x\in \Z^d}\otimes\,  \text{\normalfont Binomial}(\alpha_x, \bar \rho(\tfrac{x}{N}))
	\end{equation}
is consistent with the continuous profile $\bar \rho$ (Definition \ref{definition consistency }), thus, satisfying the assumption of Theorem \ref{theorem HDL}.
\end{proposition}
\begin{proof}Note that, for all realizations of the environment $\boldsymbol \alpha$, $N \in \N$ and $x \in \Z^d$, one has
	\begin{equation}
		\E^{\boldsymbol \alpha}_{\nu_N^{\boldsymbol \alpha,\bar \rho}}\left[\eta(x)\right]=\alpha_x\bar \rho(\tfrac{x}{N})\qquad \text{and}\qquad \E^{\boldsymbol \alpha}_{\nu_N^{\boldsymbol \alpha,\bar \rho}}\left[\left(\eta(x)-\alpha_x\bar\rho(\tfrac{x}{N}) \right)^2 \right] = \alpha_x\bar\rho(\tfrac{x}{N})\left(1-\bar\rho(\tfrac{x}{N})\right)\ .
	\end{equation}
	Hence, by Chebyshev's inequality, for all $\delta > 0$ and $G \in \mathscr S(\R^d)$, 
	\begin{equation}
		\nu_N^{\boldsymbol \alpha,\bar \rho}\left(  \left\{ \eta \in \mathcal X^{\boldsymbol \alpha}:\  \left| \frac{1}{N^d} \sum_{x \in \Z^d} G(\tfrac{x}{N}) \left(\eta(x) -  \alpha_x\bar\rho(\tfrac{x}{N})\right) \right|\, >\, \delta\, \right\}\right) \underset{N\to \infty}\longrightarrow0 
	\end{equation}
	holds true for all  $\boldsymbol \alpha$. With the observation that, for all functions $G \in \mathscr S(\R^d)$ and continuous profiles $\bar \rho:\R^d\to [0,1]$, the product of $G$ and $\bar \rho$ is continuous, bounded and integrable, Lemma \ref{lemma:ergodic_theorem} yields the desired result for 	 $\boldsymbol \alpha \in \mathfrak A$.
	
\end{proof}

\subsection{Duality}
For all given environments $\boldsymbol \alpha$, 
$\SEPa$ and $\RWa$, besides being the latter a particular instance of the former when the system consists of only one particle, are connected through the  notion  of  \emph{stochastic duality}, or, shortly, \emph{duality}. This notion occurs in various contexts (see, e.g., \cite{liggett_interacting_2005-1}) and, in the particular case of interacting particle systems, turns useful when quantities of a many-particle system may be studied in terms of quantities of a simpler, typically a-few-particle, system. Moreover, when this duality relation is established between two copies of the same Markov process, one speaks about \emph{self-duality}.

$\SEPa$ is a self-dual Markov process, meaning that there exists a function $\Da : \Xaf \times \Xa \to  \R$ (with $\Xaf$ being the subset of configurations in $\Xa$ with finitely-many particles), called \emph{self-duality function}, given by
\begin{equation*}
\Da(\xi,\eta):=\prod_{x \in \Z^d} \frac{\eta(x)!}{(\eta(x)-\xi(x))!}\frac{(\alpha_x-\xi(x))!}{\alpha_x!}\ind_{\{\xi(x)\leq \eta(x)\}}\ ,
\end{equation*}
for which the following self-duality relation holds: for all $\xi\in\Xaf$ and $ \eta \in \Xa$, 
\begin{equation}
\La \Da(\cdot,\eta)(\xi)=\La \Da(\xi,\cdot)(\eta)\ 	.
\end{equation}	
In particular, the l.h.s.\	 corresponds to apply the generator $\La$ to the function
$D(\cdot,\eta)$ and evaluate the resulting function at $\xi$; similarly for the r.h.s..
This property was proven for the first time in \cite{schutz_non-abelian_1994} for the homogeneous partial exclusion, i.e., for $\alpha_x = m \in \N$ for all $x \in \Z^d$, (see also \cite{giardina_duality_2009}) and extends to the random environment context. 

We are interested in a particular instance of this self-duality property, namely when the dual configuration  consists in a single particle configuration, i.e., $\xi=\delta_x$ for some $x \in \Z^d$. In this case the function $\Da(\delta_x,\eta) \eqqcolon \Da(x,\eta)$ reads
\begin{equation}\label{equation duality function}
\Da(x,\eta)=\tfrac{\eta(x)}{\alpha_x}
\end{equation}
and the self-duality relation reduces to
\begin{equation}\label{equation duality}
\Aa \Da(\cdot,\eta)(x)=\La \Da(x,\cdot)(\eta)\ ,
\end{equation}
which may be checked by a straightforward computation. Relation \eqref{equation duality} has to be interpreted as a duality relation between $\SEPa$ and $\RWa$ with duality function $\Da$ given in \eqref{equation duality function}. 

Notice that the generator $\Aa$ is, in view of Assumption \ref{definition  disorder}, a bounded operator on both 	  Banach spaces $\ell^\infty(\Z^d,\boldsymbol \alpha)$ and  $\ell^1(\Z^d,\boldsymbol \alpha)$, where $\boldsymbol \alpha$ plays the role of reference measure on $\Z^d$  assigning	 to each site  $x\in\Z^d$ the positive value $\alpha_x$. Likewise, $\Aa$ is a bounded operator on the weighted Hilbert space $\ell^2(\Z^d,\boldsymbol \alpha)$ whose inner product is defined as 	
\begin{equation}\label{equation inner product}
\langle f, g\rangle:=\sum_{x\in \Z^d}f(x)\, 	g(x)\, \alpha_x\ .
\end{equation}
With a slight abuse of notation, we continue to use  $\langle \cdot,\cdot\rangle$ also for the bilinear map on $\ell^1(\Z^d,\boldsymbol \alpha) \times \ell^\infty(\Z^d,\boldsymbol \alpha)$ defined by the r.h.s.\ of \eqref{equation inner product}; moreover, we let $A_{\boldsymbol \alpha}$ and $\{S_t^{\boldsymbol \alpha},\ t \geq 0\}$ denote the generator and corresponding semigroup associated to $\RWa$, indistinguishably of the Banach space they act on.
	
As it follows from a detailed balance relation, $\RWa$ is reversible with respect to the weighted counting measure $\boldsymbol \alpha$. More precisely, 
$\Aa$ is self-adjoint in $\ell^2(\Z^d,\boldsymbol \alpha)$ and, moreover, for all $f \in \ell^1(\Z^d,\boldsymbol \alpha)$ (resp.\ $\ell^2(\Z^d,\boldsymbol \alpha)$) and $g \in \ell^\infty(\Z^d,\boldsymbol \alpha)$ (resp.\ $\ell^2(\Z^d,\boldsymbol \alpha)$) and for all $t \geq 0$, we have	
\begin{align}\label{eq:self-adjoint}
\langle \Sa_t f, g\rangle=\langle f, \Sa_t g\rangle\ 	,
\end{align}
or, equivalently,
\begin{equation}\label{eq:detailed_balance_probabilities}
	\alpha_x\, p^{\boldsymbol \alpha}_t(x,y) = \alpha_y\, p^{\boldsymbol \alpha}_t(y,x)\ ,\qquad x, y\in \Z^d\ ,\ t \geq 0\ ,
\end{equation}
for the corresponding transition probabilities.

\subsection{Strategy of the proof}\label{section strategy proof}
The self-duality relation \eqref{equation duality}  suggests that the limiting collective behavior of the particle density is connected to the limiting behavior of the diffusively rescaled $\RWa$.
Let us  describe the strategy of the proof of our main result and the role of this connection.
\subsubsection{Mild solution representation}\label{section:mild_solution}
As a first observation, by following closely  \cite{nagy_symmetric_2002} and \cite{faggionato_bulk_2007}, for all realizations of the environment $\boldsymbol \alpha$, we apply Dynkin's formula to the bounded cylindrical functions $\{\Da(x,\cdot\,) : \Xa \to \R \}_{x \in\Z^d}$ given in \eqref{equation duality function}: for all initial configurations $\eta \in \Xa$,	 we have
\begin{equation}\label{dynkin1}	
	 \Da(x,\eta_t)=\Da(x,\eta)+\int_0^t\La \Da(x,\cdot\,)(\eta_s)\, \dd s+	 M^{\boldsymbol \alpha}_t(x)\ ,\qquad x \in \Z^d\ ,\ t \geq 0\ ,
\end{equation}
where $\{M^{\boldsymbol \alpha}_t(x), t \geq 0 \}_{x \in \Z^d}$ is a family of martingales w.r.t.\ the natural filtration of the process  whose joint law is characterized in terms of their predictable quadratic covariations (see \eqref{eq:predictable_cov1}--\eqref{eq:predictable_cov2} below; for an explicit construction  of these martingales, see Appendix \ref{section:ladder} below).
We remark that in \eqref{dynkin1} above $\La$ acts on the function $\Da(\cdot,\cdot)$ w.r.t.\  the $\eta$-variables.
We recall from \eqref{equation duality} that the function  $\Da: \Z^d \times \Xa\to \R$ of the joint system  is a duality function between $\SEPa$ and $\RWa$. Hence, by using 	\eqref{equation duality}, we rewrite  \eqref{dynkin1} as
\begin{equation}
	\Da(x,\eta_t)=\Da(x,\eta)+\int_0^t\Aa \Da(\cdot,\eta_s)(x)\, \dd s + 	 M^{\boldsymbol \alpha}_t(x)\ ,\qquad x \in \Z^d\ ,\ t\geq 0\ ,
\end{equation}
yielding a  system (indexed by $x \in \Z^d$) of linear -- in the drift --  stochastic integral equations. As a consequence, the solution of this system may be represented as a mild solution by considering the semigroup $\{\Sa_t, t \geq 0 \}$ associated to the generator $\Aa$ of $\RWa$, i.e., we have
\begin{align}\label{mild solution}
\Da(x,\eta_t)=\Sa_t\Da(\cdot,\eta)(x)+\int_0^t \Sa_{t-s}\, \dd M^{\boldsymbol \alpha}_s(x)\ ,\qquad x \in \Z^d\ ,\ t \geq 0\ ,
\end{align} 
where 
\begin{equation}\label{eq:martingale_occupation_variable}
	\int_0^t \Sa_{t-s}\, \dd M^{\boldsymbol \alpha}_s(x):=\int_0^t \sum_{y\in \Z^d} P^{\boldsymbol \alpha}(X^{\boldsymbol \alpha,x}_{t-s}=y)\, \dd M^{\boldsymbol \alpha}_s(y)
	\end{equation} (for a definition of $X^{\boldsymbol \alpha,x}$ and its law, see the end of Section \ref{section:model}; for a proof of the absolute convergence of the latter infinite sum, we refer the reader to  Lemma \ref{lemma mild solution exclusion} below).

Combining  the definitions  \eqref{equation empirical density field}--\eqref{eq:density_fieldsG} and \eqref{equation duality function} with the  mild solution representation in \eqref{mild solution},  we rewrite the empirical density fields, for all test functions $G \in \mathscr S(\R^d)$, as follows:
\begin{align*}
\mathsf X^N_t(G) &= \frac{1}{N^d} \sum_{x \in \Z^d} G(\tfrac{x}{N})\, \Da(x,\eta_{tN^2})\, \alpha_x\\
&= \frac{1}{N^d} \sum_{x \in \Z^d} G(\tfrac{x}{N})\, \Sa_{tN^2}\Da(\cdot,\eta_0)(x)\, \alpha_x+  \frac{1}{N^d} \sum_{x \in \Z^d} G(\tfrac{x}{N})\left(\int_0^{tN^2} \Sa_{tN^2-s}\, \dd M^{\boldsymbol \alpha}_s(x)\right) \alpha_x\ .
\end{align*}
Furthermore, because both $\Aa$ and the corresponding semigroup are  self-adjoint in $\ell^2(\Z^d,\boldsymbol \alpha)$  (see \eqref{eq:self-adjoint}), we obtain:
\begin{align}\nonumber
\mathsf X^N_t(G)
& = \frac{1}{N^d} \sum_{x \in \Z^d} S^{N,\boldsymbol \alpha}_{tN^2}G(\tfrac{x}{N})\, \Da(x,\eta_0)\, \alpha_x+  \frac{1}{N^d} \sum_{x \in \Z^d} \left(\int_0^{tN^2} S^{N,\boldsymbol \alpha}_{tN^2-s}G(\tfrac{x}{N})\, \dd M^{\boldsymbol \alpha}_s(x)\right) \alpha_x\\
&=\ \nonumber  \frac{1}{N^d} \sum_{x \in \Z^d} S^{N,\boldsymbol \alpha}_{tN^2}G(\tfrac{x}{N})\, \eta_0(x)	+  \frac{1}{N^d} \sum_{x \in \Z^d} \left(\int_0^{tN^2} S^{N,\boldsymbol \alpha}_{tN^2-s}G(\tfrac{x}{N})\, \dd M^{\boldsymbol \alpha}_s(x)\right) \alpha_x\\
\label{elll}
&=\ \mathsf X^N_0(S^{N,\boldsymbol \alpha}_{tN^2}G)+ \int_0^t \dd	\mathsf  M^N_s(S^{N,\boldsymbol \alpha}_{tN^2-s}G)\  ,
\end{align}
where we adopted the shorthand, for all $G \in \mathscr S(\R^d)$,
\begin{equation}
\int_0^t \dd	\mathsf  M^N_s(S^{N,\boldsymbol \alpha}_{tN^2-s}G)\ :=\		\frac{1}{N^d} \sum_{x \in \Z^d} \left(\int_0^{tN^2} S^{N,\boldsymbol \alpha}_{tN^2-s}G(\tfrac{x}{N})\, \dd M^{\boldsymbol \alpha}_s(x)\right) \alpha_x\ ,	
\end{equation}
with
\begin{equation}\label{eq:semigroupsN}
S^{N,\boldsymbol \alpha}_tG(\tfrac{x}{N})\ :=\ \Sa_tG(\tfrac{\cdot}{N})(x)\ ,\qquad x \in \Z^d\ ,\ t \geq 0\ .
\end{equation}
Hence, we obtain in \eqref{elll} the same decomposition as in, e.g., \cite{nagy_symmetric_2002, faggionato_bulk_2007, faggionato_hydrodynamic_2009, redig_symmetric_2020}, in which the empirical density field is written as a sum of its expectation (the first term on the r.h.s.\ of \eqref{elll}), and \textquotedblleft noise\textquotedblright (the second term), which is not a martingale. 

\subsubsection{From the arbitrary starting point invariance principle towards the path space hydrodynamic limit}

As in those works, our first aim is to prove that, for $\mathcal P$-a.e.\ $\boldsymbol\alpha$, the finite-dimensional distributions of the empirical density fields converge in probability to those of the solution of the hydrodynamic equation \eqref{heat equation}. Moreover,   since convergence in probability of finite-dimensional distributions is implied by the convergence in probability of single marginals, it suffices to prove convergence of one-dimensional distributions. In particular,  we will  show in Section \ref{section proof HDL} below that,   for all $G \in \mathscr S(\R^d)$, $t \geq 0$ and  $\delta>0$,
	\begin{align}\label{equation variance va a 0}
		\P^{\boldsymbol \alpha}_{\nu_N^{\boldsymbol \alpha}}\left(\left|\int_0^{tN^2} \dd\mathsf M^N_s(S^{N,\boldsymbol \alpha}_{tN^2-s}G)  \right|> \delta\right ) \underset{N \to \infty}\longrightarrow\ 0\ 
	\end{align}
	holds for all environments $\boldsymbol \alpha$, and that (recall \eqref{eq:pi_0})
	\begin{align}\label{equation convergence finite dim dist} \P^{\boldsymbol \alpha}_{\nu_N^{\boldsymbol \alpha}}\left(\left|\mathsf X^N_0(S^{N,\boldsymbol \alpha}_{tN^2}G) - \pi^{\bar \rho}(\mathcal S^\varSigma_t G) \right| > \delta  \right)\underset{N \to \infty}\longrightarrow\ 0
	\end{align}
	holds for $\mathcal P$-a.e.\ environment
	 $\boldsymbol \alpha$.
	Hence, provided that $\{\mathsf X^N_t,\, t \geq 0\}$ is relatively compact in $\mathcal D([0,\infty),\mathscr S'(\R^d))$ and that all limit points belong to $\mathcal C([0,\infty),\mathscr S'(\R^d))$, 
	 \eqref{equation variance va a 0}--\eqref{equation convergence finite dim dist} and the uniqueness result in Remark \ref{remark:uniquness} would then yield a quenched (w.r.t.\ the environment law $\mathcal P$) convergence in probability of finite-dimensional distributions for the empirical density fields.

More specifically, the convergence in \eqref{equation variance va a 0}  (whose proof is close in spirit to that in all other related works) relies on Chebyshev's inequality  and  the uniform upper bound \eqref{equation uniform ellipticity} on the environment $\boldsymbol \alpha$.  This result is established in Section \ref{section finite dimensional distributions} below.  	
For what concerns  \eqref{equation convergence finite dim dist}, as done in the aforementioned references, the idea is to go through a homogenization result which ensures convergence -- in a sense to be made precise -- of semigroups  for $\mathcal P$-a.e.\ $\boldsymbol \alpha$. In particular,   provided $\boldsymbol \alpha$ is an environment for which  the following $L^1$-convergence
\begin{align}\label{eq:l1_convergence}
\frac{1}{N^d} \sum_{x \in \Z^d} \left|S^{N,\boldsymbol \alpha}_{tN^2}G(\tfrac{x}{N}) - \mathcal S^\varSigma_t G(\tfrac{x}{N})\right|\alpha_x\ \underset{N\to \infty}\longrightarrow\ 0\ ,\qquad t \geq 0\ ,
\end{align}
holds for all $G\in \mathscr S(\R^d)$, 
 Markov's inequality, the uniform boundedness of the occupation variables $\{\eta(x),\ x \in \Z^d\}$ and \eqref{eq:l1_convergence} yield
\begin{equation}\label{equation astqip}
\P^{\boldsymbol \alpha}_{\nu_N^{\boldsymbol \alpha}}\left(\left|\mathsf X^N_0(S^{N,\boldsymbol \alpha}_{tN^2}G) - \mathsf X^N_0(\mathcal S^\varSigma_tG) \right| > \delta	  \right)\underset{N \to \infty}\longrightarrow\ 0\ ,\qquad t \geq 0\ ,
\end{equation}
for that same environment $\boldsymbol \alpha$ and all test functions $G \in \mathscr S(\R^d)$.	 By combining \eqref{equation astqip} -- which will hold for $\mathcal P$-a.e.\ $\boldsymbol \alpha$ -- 	with the assumption of $\mathcal P$-a.s.\  consistency of initial conditions (see the statement of Theorem \ref{theorem HDL} and Definition \ref{definition consistency }), we obtain \eqref{equation convergence finite dim dist} for $\mathcal P$-a.e.\ $\boldsymbol \alpha$. 
All the details of the proof of \eqref{equation convergence finite dim dist} may be found in Proposition \ref{proposition finite-dim distributions} below.

In view of these considerations, the proof of convergence of the finite dimensional distributions of the empirical density fields boils down to show  \eqref{equation variance va a 0} and \eqref{eq:l1_convergence}. Several methods have been developed in, e.g., \cite{nagy_symmetric_2002, faggionato_bulk_2007, faggionato_hydrodynamic_2009, faggionato_stochastic_2020} to obtain \eqref{eq:l1_convergence}. The road we follow here is to derive \eqref{eq:l1_convergence} from quenched invariance principle results for random conductance models ($\RCM$) (see, e.g., \cite{biskup_recent_2011}) in the following two steps:

\begin{enumerate}[label={\normalfont (\roman*)},ref={\normalfont (\roman*)}]
	\item By viewing our random walks $\RWa$ as  random time changes of suitable $\RCM$, we derive from well-known analogous results in the context of $\RCM$, a \emph{quenched invariance principle} for  the random walk $\RWa$ \emph{started from the origin}.
	
	\item By means of the space-time H\"older equicontinuity of the semigroups $\{S^{N,\boldsymbol \alpha}_t,\, t \geq 0\}_{N \in \N}$ (see \eqref{eq:semigroupsN} for its definition), heat kernel upper bounds and building on the ideas in \cite[Appendix	 A.2]{chen_quenched_2015}, we obtain: for $\mathcal P$-a.e.\ realization of the environment $\boldsymbol \alpha$, 
	\begin{equation}
	\label{eq:uniform_convergence_semigroups0}
		\parbox{\dimexpr\linewidth-4em}{%
			\strut For all $T > 0$ and $G \in \mathcal C_0(\R^d)$, 	 
	\begin{equation*}
	\sup_{t \in [0,T]} \sup_{x \in \Z^d}
	\left|S^{N,\boldsymbol \alpha}_{tN^2}G(\tfrac{x}{N}) - \mathcal S_t^\varSigma G(\tfrac{x}{N})\right|\ \underset{N\to \infty}\longrightarrow\ 0\ ,
	\end{equation*}		
holds true.
	\strut
}
\end{equation}
\end{enumerate}
Relating the above convergence of Markov semigroups  to the weak convergence in path-space of the corresponding Markov processes, it is straightforward to check that  \eqref{eq:uniform_convergence_semigroups0} implies the weak convergence of the finite dimensional distributions of $\RWa$   with \emph{arbitrary starting positions}, i.e., for $\mathcal P$-a.e.\ $\boldsymbol \alpha$, 
\begin{equation*}
	\parbox{\dimexpr\linewidth-4em}{%
		\strut
		For all  $u\in \R^d$ and for any sequence of points $\{x_N\}_{N \in \N} \subseteq \Z^d$ such that $\tfrac{x_N}{N}\to u$ as $N \to \infty$, 
		\begin{equation*}
				E^{\boldsymbol \alpha}\left[G_1\left(\tfrac{X^{\boldsymbol \alpha,x_N}_{t_1N^2}}{N}\right)\cdots G_n\left(\tfrac{X^{\boldsymbol \alpha,x_N}_{t_n N^2}}{N}\right) \right]\underset{N\to \infty}\longrightarrow  \mathtt E\left[G_1\left(B^{\varSigma,u}_{t_1}\right)\cdots G_n\left(B^{\varSigma,u}_{t_n}\right) \right]
		\end{equation*} 
	holds true   for all $n \in \N$, $0 \leq t_1<\ldots < t_n$ and $G_1,\ldots, G_n \in \mathcal C_0(\R^d)$, where $\{B^{\varSigma, u}_t:=B^\varSigma_t+u,\ t \ge 0\}$ is the  Brownian motion introduced in Section \ref{Section HDL} started from $u \in \R^d$.
		\strut}
\end{equation*}
Moreover, as a direct consequence of the heat kernel upper bound in Proposition \ref{proposition:heat_kernel_bounds} below, the tightness of the random walks $\{\frac{1}{N}X^{\boldsymbol \alpha,x_N}_{tN^2},\, t \geq 0\}_{N \in \N}$ in $\mathcal D([0,\infty), \R^d)$ can also be derived (see, e.g., \cite[Lemma C.3]{redig_symmetric_2020}). In view of this implication, we will refer to \eqref{eq:uniform_convergence_semigroups0} as the \emph{arbitrary starting point invariance principle}. See also  Theorem \ref{theorem:uniform_convergence_semigroups} below for a slightly more precise statement regarding the convergence in \eqref{eq:uniform_convergence_semigroups0} and Remark \ref{remark: equivalence AQPQIP and convergence semigroup} below for a discussion on the equivalence between  \eqref{eq:uniform_convergence_semigroups0}  and the weak convergence in path-space of  the corresponding Markov processes; for a general result on the fact that the convergence in  \eqref{eq:uniform_convergence_semigroups0} implies convergence of the corresponding Markov processes we refer the interested reader to \cite[Theorem 4.29]{kurtz_semigroups_1975}.

 As shown in Corollary \ref{corollary:l1} below, the convergence in \eqref{eq:uniform_convergence_semigroups0} implies, in particular, for $\mathcal P$-a.e.\ $\boldsymbol \alpha$ and for all  $T > 0$ and $G \in \mathscr S(\R^d)$, 
\begin{align}\label{eq:l1sup_convergence0}
\sup_{t \in [0,T]}\frac{1}{N^d} \sum_{x \in \Z^d} \left|S^{N,\boldsymbol \alpha}_{tN^2}G(\tfrac{x}{N}) - \mathcal S^\varSigma_t G(\tfrac{x}{N})\right|\alpha_x\ \underset{N\to \infty}\longrightarrow\ 0\ .
\end{align}
Note that the above convergence differs from \eqref{eq:l1_convergence} by the uniformity of the convergence over bounded intervals of times. 

The results \eqref{eq:uniform_convergence_semigroups0} and \eqref{eq:l1sup_convergence0} are stronger than what is strictly needed  for the proof of convergence of finite dimensional distributions of the empirical density fields, but they  turn out to be very useful  in the proof of relative compactness of the probability distributions of
\begin{align}\label{eq:sequence_fields}
\left\{\mathsf X^N_t,\, t \in [0,T] \right\}_{ N \in \N}
\end{align}
in $\mathcal D([0,T],\mathscr S'(\R^d))$. Indeed, because the random walk $\RWa$ semigroups enter in the decomposition of the empirical density fields, it has to be expected that some sort of equicontinuity in time of such semigroups is needed for the  sequence \eqref{eq:sequence_fields} to be tight. This intuition can be made rigorous by means of a combination of the tightness criteria developed in \cite[Theorem 4.1]{mitoma_tightness_1983} and \cite[Appendix B]{redig_symmetric_2020}, which apply directly to the empirical density fields decomposed as mild solutions. We refer the reader	 to  Section \ref{section tightness} below for all the details on the proof of tightness.

\section{Arbitrary starting point quenched invariance principle}\label{section invariance principles}
This section is devoted to the proof of a quenched homogenization result for the dual random walk in random environment $\boldsymbol \alpha$, $\RWa$ with generator $A_{\boldsymbol \alpha}$ given in \eqref{equation generator random walk in disorder} and corresponding semigroup $\{\Sa_t,\, t \geq 0\}$. More precisely, we will prove the following theorem:
\begin{theorem}[{Arbitrary starting point quenched invariance principle}]\label{theorem:uniform_convergence_semigroups} 
There exists a measurable subset $\mathfrak B \subseteq \{1,\ldots, \mathfrak c\}^{\Z^d}$ (defined in \eqref{eq:B} below) with $\mathcal P(\mathfrak B)=1$ and such that, for all  environments $\boldsymbol \alpha \in \mathfrak B$, for all $T > 0$ and  $G \in \mathcal C_0(\R^d)$, 
	\eqref{eq:uniform_convergence_semigroups0} holds, i.e., 	
	\begin{equation}\label{eq:uniform_convergence_semigroups}
	\sup_{t \in [0,T]} \sup_{x \in \Z^d}
	\left|S^{N,\boldsymbol \alpha}_{tN^2}G(\tfrac{x}{N}) - \mathcal S_t^\varSigma G(\tfrac{x}{N})\right|\ \underset{N\to \infty}\longrightarrow\ 0\ .
	\end{equation}
\end{theorem}
The proof of the above theorem is deferred to Section \ref{section:proof_uniform_convergence_semigroups} below, and  goes through the proof of three intermediate results: the quenched invariance  principle for the random walk started from the origin (see Proposition \ref{proposition:QFCLT_site_disorder} in Section \ref{section: QIP origin} below), the space-time equicontinuity of the random walk semigroups (see Proposition \ref{prop:holder cont} in Section \ref{section:analytical} below) and  heat kernel upper bounds (see Proposition \ref{proposition:heat_kernel_bounds} in Section \ref{section:analytical} below).

As a consequence of Theorem \ref{theorem:uniform_convergence_semigroups} above and Lemma \ref{lemma:ergodic_theorem}, and recalling from there the  characterizations of the subsets   $\mathfrak B$ and $\mathfrak A \subseteq \{1,\ldots, \mathfrak c\}^{\Z^d}$, respectively, we obtain:
\begin{corollary}\label{corollary:l1}
For all environments $\boldsymbol \alpha \in \mathfrak A \cap \mathfrak B$, for all $T >0$ and $G \in \mathscr S(\R^d)$, \eqref{eq:l1sup_convergence0} holds, i.e.,		
	\begin{equation}\label{eq:l1sup_convergence}
		\sup_{t \in [0,T]}\frac{1}{N^d} \sum_{x \in \Z^d} \left|S^{N,\boldsymbol \alpha}_{tN^2}G(\tfrac{x}{N}) - \mathcal S^\varSigma_t G(\tfrac{x}{N})\right|\alpha_x \underset{N\to \infty}\longrightarrow 0\ .	
	\end{equation}
\end{corollary}
The proof of the above corollary -- whose main ideas are adapted from \cite[Proposition 5.3]{redig_symmetric_2020} --	 is postponed to Appendix \ref{appendix:ergodic} below.

\subsection{Quenched invariance principle for $\RWa$  starting from the origin}\label{section: QIP origin} For all realizations $\boldsymbol \alpha$ of the environment, 
the random walk $\RWa$, $X^{\boldsymbol \alpha,0} = \{X^{\boldsymbol \alpha,0}_t,\, t \geq 0 \}$ -- with generator given in \eqref{equation generator random walk in disorder} and with the origin of $\Z^d$ as starting position -- can be viewed as a random time change of a specific RCM, i.e., the continuous-time  random walk $X^{\boldsymbol \omega,0}=\{X^{\boldsymbol \omega,0}_t,\, t \geq 0\}$, abbreviated by $\RWo$ and with law $P^{\boldsymbol \omega}$ (and corresponding expectation $E^{\boldsymbol \omega}$), starting  from the origin of $\Z^d$ and evolving on $\Z^d$ according to the generator given by
\begin{equation}\label{equation generator random conductance}
\Ao f(x):=\sum_{\substack{y\in \Z^d\\|y-x|=1}}\omega_{xy}\left(f(y)-f(x)\right)\ ,\qquad x \in \Z^d\ ,
\end{equation}
where $f:\Z^d\to \R$ is a bounded  function and
\begin{align}\label{equation conductances connected to alpha}
\omega_{xy}\coloneqq\alpha_x\alpha_y\ ,\qquad \forall\, x, y \in \Z^d\ \text{such that}\ |x-y|=1 \ .	
\end{align}
Indeed, when in  position $x \in \Z^d$, the walk $X^{\boldsymbol \alpha,0}$ spends there an exponential holding time with parameter $\lx$ given by  
\begin{align}\label{equation holdin time alpha}
\lx=\sum_{\substack{y\in \Z^d\\ |y-x|=1}}\alpha_y\ ,
\end{align} and then jumps to a neighbor of $x$, say $z$, with probability $r_{\boldsymbol \alpha}(x,z)$ given by 
\begin{align}\label{equation jump probability}
r^{\boldsymbol \alpha}(x,z)=\frac{\alpha_z}{\lx}\ .
\end{align}
The corresponding quantities \eqref{equation holdin time alpha} and \eqref{equation jump probability} for the walk $X^{\boldsymbol \omega,0}$  are given, respectively, by 
\begin{equation}\label{equation holding time omega}
\lox=\sum_{\substack{y\in \Z^d\\ |y-x|=1}}\alpha_x\alpha_y=\alpha_x \lx \ ,
\end{equation}
and
\begin{align}\label{equation transition probabilities random conductance}
r^{\boldsymbol \omega}(x,z)=\frac{\alpha_x\alpha_z}{\alpha_x\lx}=r^{\boldsymbol \alpha}(x,z)\ .
\end{align}
Hence, if we define the random time change $\{R(t), t \geq 0\}$ by 
\begin{equation}\label{equation random time change A}
R(t)\coloneqq \int_0^t \alpha_{X^{\boldsymbol \omega,0}_s}\, \dd s\ ,
\end{equation} then, in law, 
\begin{align*}
\{X^{\boldsymbol \omega,0}_{R^{-1}(t)},\, t \geq 0\} =\{X^{\boldsymbol \alpha,0}_t,\, t \geq 0\}\ ,
\end{align*}  
where $R^{-1}$ is the inverse of the continuous piecewise linear and increasing bijection $R:[0,\infty)\to[0,\infty)$.  

\

In what follows, we let $\Omega$ denote the space of all  conductances $\boldsymbol \omega$ with $\omega_{xy}\in\{1,...,\mathfrak c^2\}$ endowed with the Borel $\sigma$-algebra induced by the discrete topology. Recall the definition of $\mathcal P$ in Assumption \ref{definition  disorder}. We then let $\mathcal Q$ be the probability measure on  $\Omega$ for which, for all measurable $\mathcal U\subseteq \Omega$, 
\begin{align}\label{equation QP}
\mathcal Q(\mathcal U)=\mathcal P\left(\boldsymbol \alpha \in  \{1,\ldots, \mathfrak c\}^{\Z^d}  : \begin{array}{c}\exists\, \boldsymbol \omega \in \mathcal U\ \text{s.t.}\  \omega_{xy}=\alpha_x\alpha_y\\ \forall\  x,y\in \Z^d\ \text{with}\ |x-y|=1\end{array}\right)\ .
\end{align}
We remark that the measure $\mathcal Q$ inherits the invariance and ergodicity under space translations from $\mathcal P$ (see Assumption \ref{definition  disorder}). 
We then have the following result, taken  from \cite[Theorem 1.1 and Remark 1.3]{SidoraviciusSznitman}. 
\begin{theorem}[Quenched invariance principle for $\RWo$ started from the origin \cite{SidoraviciusSznitman}]\label{theorem QIP for rundom conductances }
	The quenched  invariance principle holds for the random walk $\RWo$ started from the origin with a limiting non-degenerate covariance matrix $\varLambda$, i.e., for $\mathcal Q$-a.e.\ environment $\boldsymbol \omega$ and for all $T > 0$, the following convergence in law in the Skorokhod space $\mathcal D([0,T],\R^d)$ holds
	\begin{align*}
	\left\{\frac{X^{\boldsymbol \omega,0}_{tN^2}}{N},\,  t \in [0,T] \right\}\ \underset{N\to \infty}\Longrightarrow\ \left\{B^{\varLambda}_t,\, t\in [0,T] \right\}\ ,
	\end{align*}
 where the r.h.s.\ is a Brownian motion on $\R^d$ starting at the origin with a non-degenerate covariance matrix $\varLambda \in \R^{d\times d}$ independent of the  realization of the environment $\boldsymbol \omega$.		
\end{theorem}
We remark that   \cite{SidoraviciusSznitman} and \cite{mathieu_quenched_2008} were the  first two works in which the quenched invariance principle for RCM with ergodic and uniformly elliptic conductances was proven for any dimension $d \geq 1$. 	We refer to, e.g.,  \cite{BergerBiskup, mathieu_quenched_2007, biskup_invariance_2019, andres_invariance2015, bella_quenched_2019} as a partial list for further results in which the uniform ellipticity assumption on the conductances has been replaced by more general conditions on the conductance moments.

In order to get the quenched  invariance principle for the random walk $\RWa$,   we only need  to check that the random time change defined in \eqref{equation random time change A} properly rescales. In the proof of the following result,  we follow closely Section 6.2 in \cite{andres_invariance2013}. 

\begin{proposition}[{Quenched invariance principle for $\RWa$ started from the origin}]\label{proposition:QFCLT_site_disorder}
	The quenched  invariance principle holds for  the random walk $\RWa$ started from the origin with a limiting non-degenerate covariance matrix $\varSigma:=\frac{1}{\E_{\mathcal P}\left[\alpha_0\right]}\varLambda$. Here $\varLambda$ is the covariance matrix appearing in Theorem \ref{theorem QIP for rundom conductances }. In particular, the covariance matrix $\varSigma$ does not depend on the specific realization of the environment $\boldsymbol \alpha$, but only on the law $\mathcal P$.
\end{proposition}
\begin{remark}\label{remark:B}
	For later purposes, we define 
	\begin{equation}\label{eq:B}
		\mathfrak B\coloneqq \left\{\boldsymbol \alpha \in \{1,\ldots, \mathfrak c\}^{\Z^d}: \text{the invariance principle for $\RWa$ in Proposition \ref{proposition:QFCLT_site_disorder} holds}\right\}\ .	
	\end{equation}
\end{remark}

\begin{proof}
	Consider the random walk $X^{\boldsymbol \omega,0}$ starting from the origin and the corresponding process of the environment $\boldsymbol \alpha$ as seen from the random walk $X^{\boldsymbol \omega,0}$, i.e., 
	\begin{equation}\label{eq:environment_process}
		\left\{\tau_{X^{\boldsymbol \omega,0}_t}\,	\boldsymbol \alpha,\ t \geq 0\right\} \subseteq \{1,\ldots,\mathfrak c\}^{\Z^d}\ .	
	\end{equation}
By our Assumption \ref{definition  disorder} and \cite[Lemma 4.3]{de_masi_invariance_1989},  $\mathcal P$ is an invariant (actually reversible) and ergodic law for the process in \eqref{eq:environment_process}.
Hence, recalling the random time change $\{R(t),\, t \geq 0 \}$ defined in \eqref{equation random time change A},  Birkhoff's ergodic theorem for the process in \eqref{eq:environment_process} yields, for $\mathcal P$-a.e.\ environment $\boldsymbol \alpha$, 
	\begin{equation}\label{equation ergodic theorem time change 1}
	\lim_{t\to \infty} \frac{R(t)}{t} =\E_{\mathcal P}\left[\alpha_0\right]\ .
	\end{equation}
	Because $R: [0,\infty) \to [0,\infty)$ is a strictly increasing bijection, \eqref{equation ergodic theorem time change 1} is, in turn,  equivalent to 
	\begin{equation}\label{equation ergodic theorem time change}
	\lim_{t\to \infty} \frac{R^{-1}(t)}{t}=\frac{1}{\E_{\mathcal P}\left[\alpha_0\right]}\ .
	\end{equation}
	The conclusion of the theorem follows from the argument in Section 6.2 in \cite{andres_invariance2013} as soon as we prove that, for all $t>0$ and $\epsilon>0$,   we have, for $\mathcal P$-a.e.\ $\boldsymbol\alpha$ (recall from \eqref{equation conductances connected to alpha} that $\boldsymbol \omega=\boldsymbol \omega(\boldsymbol \alpha)$ with		 $\omega_{xy}=\alpha_x\alpha_y$ for all $x,  y \in \Z^d$ with $|x-y|=1$),
	\begin{equation} \label{equation converges to 0 of increments with random time change}
	\limsup_{N\to\infty}  P^{\boldsymbol \omega} \left( \left|\frac{X^{\boldsymbol \omega,0}_{R^{-1}(tN^2)}-X^{\boldsymbol \omega,0}_{\frac{1}{\E_{\mathcal P}\left[\alpha_0\right]}tN^2}}{N}\right|>\epsilon  \right)=0\ ,
	\end{equation}
	where  $P^{\boldsymbol \omega}$ denotes the law of $X^{\boldsymbol \omega}$.
		
	We are, thus, left with the proof of \eqref{equation converges to 0 of increments with random time change}. Fix $t>0$ and $\epsilon > 0$. Then, for all $\delta>0$, we have
	\begin{align}\label{equation x}
		\nonumber
	&P^{\boldsymbol \omega} \left( \left|\frac{X^{\boldsymbol \omega,0}_{R^{-1}(tN^2)}-X^{\boldsymbol \omega,0}_{\frac{1}{\E_{\mathcal P}\left[\alpha_0\right]}tN^2}}{N}\right|>\epsilon  \right) \\
	\nonumber
	&\quad \leq P^{\boldsymbol \omega} \left( \left|\frac{X^{\boldsymbol \omega,0}_{R^{-1}(tN^2)}-X^{\boldsymbol \omega,0}_{\frac{1}{\E_{\mathcal P}\left[\alpha_0\right]}tN^2}}{N}\right|>\epsilon, \left|  \frac{R^{-1}(tN^2)}{N^2}- \frac{t}{\E_{\mathcal P}\left[\alpha_0\right]}\right | \leq \delta  \right) \\
	&\quad +P^{\boldsymbol \omega} \left( \left|  \frac{R^{-1}(tN^2)}{N^2}- \frac{t}{\E_{\mathcal P}\left[\alpha_0\right]}\right | > \delta   \right).
	\end{align}
	The second term on the r.h.s.\ of \eqref{equation x}  goes to zero as $N\to \infty$ by \eqref{equation ergodic theorem time change}, while the first term is bounded above by  
	\begin{equation}\label{eq:upper_bound}
	 P^{\boldsymbol \omega} \left( \sup_{\substack{|s-r|\leq \delta\\r,s\leq T}} \left|\frac{X^{\boldsymbol \omega,0}_{s N^2}-X^{\boldsymbol \omega,0}_{r N^2}}{N}\right|>\epsilon \right)\ .	
	 \end{equation}
 for some positive $T=T(t,\mathfrak c)$ independent of $N \in \N$. Due to Theorem \ref{theorem QIP for rundom conductances } and the continuity of the trajectories of the limit process, the expression in \eqref{eq:upper_bound} vanishes as $N\to \infty$ and  $\delta \to 0$, i.e.,
 \begin{equation}\label{eq:upper_bound1}
 \lim_{\delta \downarrow 0}\limsup_{N\to \infty}	P^{\boldsymbol \omega} \left( \sup_{\substack{|s-r|\leq \delta\\r,s\leq T}} \left|\frac{X^{\boldsymbol \omega,0}_{s N^2}-X^{\boldsymbol \omega,0}_{r N^2}}{N}\right|>\epsilon \right)=0\ .	
 \end{equation}  Indeed, let $\tilde X^{\boldsymbol \omega,0}=\{\tilde X^{\boldsymbol \omega,0}_t,\, t \in [0,T]\}$ denote the piecewise linear interpolation of the jump process $X^{\boldsymbol \omega,0}$. Then, due to the continuity of the trajectories of the limiting Brownian motion in Theorem \ref{theorem QIP for rundom conductances }, the same theorem holds with $\mathcal C([0,T],\R^d)$ (the Banach space of continuous functions from $[0,T]$ to $\R^d$ endowed with the supremum norm; see, e.g., \cite[Chapter 8]{billingsley_convergence_1968}) and $\tilde X^{\boldsymbol{\omega},0}$ in place of $\mathcal D([0,T],\R^d)$ and $X^{\boldsymbol \omega,0}$, respectively.  By Prohorov's theorem (see, e.g., \cite[Theorem 6.2]{billingsley_convergence_1968}) and the characterization of tightness of probability measures  on $\mathcal C([0,T],\R^d)$ (see, e.g., \cite[Theorem 8.2]{billingsley_convergence_1968}), we have, for all $\epsilon > 0$, 
 \begin{equation}\label{eq:upper_bound2}
 \lim_{\delta \downarrow 0}\limsup_{N\to \infty}	P^{\boldsymbol \omega} \left( \sup_{\substack{|s-r|\leq \delta\\r,s\leq T}} \left|\frac{\tilde X^{\boldsymbol \omega,0}_{s N^2}-\tilde X^{\boldsymbol \omega,0}_{r N^2}}{N}\right|>\epsilon \right)=0\ .	
 \end{equation}
For all $\delta > 0$ and $\epsilon>0$,  $X^{\boldsymbol \omega,0}$ being a nearest-neighbor random walk implies that
\begin{equation*}
	\limsup_{N\to \infty}	P^{\boldsymbol \omega} \left( \sup_{\substack{|s-r|\leq \delta\\r,s\leq T}} \left|\frac{\tilde X^{\boldsymbol \omega,0}_{s N^2}-\tilde X^{\boldsymbol \omega,0}_{r N^2}}{N}\right|>\epsilon \right) = \limsup_{N\to \infty}	P^{\boldsymbol \omega} \left( \sup_{\substack{|s-r|\leq \delta\\r,s\leq T}} \left|\frac{ X^{\boldsymbol \omega,0}_{s N^2}- X^{\boldsymbol \omega,0}_{r N^2}}{N}\right|>\epsilon \right)
\end{equation*}
holds true.
This and \eqref{eq:upper_bound2} yield \eqref{eq:upper_bound1}, thus,  concluding the proof of the proposition.
\end{proof}

\subsection{H\"older equicontinuity of the semigroup  and heat kernel upper bounds for $\RWa$}\label{section:analytical}
In this section, $\boldsymbol \alpha$ is an arbitrary realization of the environment.
We start by proving   that the family of semigroups corresponding to the diffusively rescaled random walks $\RWa$ are H\"{o}lder equicontinuous in both space and time variables. It is well-known (see, e.g., \cite{stroock_markov_1997, delmotte_parabolic_1999} as references in the context of graphs) that H\"older equicontinuity of solutions to parabolic partial differential equations may be derived from parabolic Harnack inequalities (see, e.g., \cite[Definition 1.6]{delmotte_parabolic_1999}). In our context, for all bounded functions $f: \Z^d \to \R $, the parabolic partial difference equation  that $S^{\boldsymbol \alpha}_\cdot f(\cdot)= \left\{S^{\boldsymbol \alpha}_t f(x),\  t \geq 0,\	 x \in \Z^d \right\}$ solves reads as follows:
\begin{equation}\label{eq: parabolic equation alpha}
\alpha_x\frac{\partial}{\partial t}\psi(t,x)=\sum_y \alpha_x\alpha_y(\psi(t,y)-\psi(t,x))\ ,\qquad t \geq 0\ ,\quad x \in \Z^d\ ,
\end{equation}
with initial condition $\psi(0,\cdot)=f$.
By applying the Moser iteration scheme as in \cite[Section 2]{delmotte_parabolic_1999}, we recover the parabolic Harnack inequality (\cite[Theorem 2.1]{delmotte_parabolic_1999}) for positive solutions of \eqref{eq: parabolic equation alpha}.  We note that $\boldsymbol \alpha$, viewed as a $\sigma$-finite measure on $\Z^d$ and due to the assumption of uniform ellipticity, plays the role of  speed measure   (cf.\ $m$ in \cite[Section 1.1]{delmotte_parabolic_1999}; see also \cite[Remark 1.5]{andres_harnack_2016} for an analogous discussion).

 In conclusion, by applying the aforementioned parabolic Harnack inequality as, e.g.,  in \cite[Proposition 4.1]{delmotte_parabolic_1999} and \cite[Theorem 1.31]{stroock_markov_1997}, we obtain the following result:

\begin{proposition}[{H\"older equicontinuity of semigroups}]\label{prop:holder cont}
	There exists $C>0$ and $\gamma>0$ such that, for all realizations $\boldsymbol \alpha$ of the environment, for all $N\in \N$ and for all  $G \in \mathcal C_0(\R^d)$, we have 
	\begin{equation}
	\left|S^{N,\boldsymbol \alpha}_{tN^2}G(\tfrac{x}{N})-S^{N,\boldsymbol \alpha}_{sN^2}G(\tfrac{y}{N})\right|\leq C\, \sup_{u\in \R^d}\left|G(u)\right| \left(    \frac{   \sqrt{|t-s|} \vee |\tfrac{x}{N}-\tfrac{y}{N}| }{\sqrt{t\wedge s}} \right)^\gamma
	\end{equation}
	for all $s, t > 0$ and $x, y \in \Z^d$.
\end{proposition}
%
The second result is an upper bound for the heat kernel of the random walk $\RWa$, i.e., 	\begin{align}\label{eq: heat kernel}
	q^{\boldsymbol \alpha}_t(x,y):=\frac{1}{\alpha_y}P^{\boldsymbol \alpha}(X^{\boldsymbol \alpha,x}_t=y)\equiv \frac{p_t^{\boldsymbol \alpha}(x,y)}{\alpha_y}\ .
\end{align}
More precisely, we need to ensure that the tails of the heat kernels satisfy a uniform integrability condition. To this aim, many results of heat kernel upper bounds which have been established in the literature, such as Gaussian upper bounds (see, e.g., \cite[Theorem 2.3]{barlow_invariance_2010}), would suffice. Here, we follow  Nash-Davies' method as in Section 3 in \cite{carlen_upper_1986} applied to our context. 	 Indeed, by \cite[Theorem 3.25]{carlen_upper_1986}, if  Nash inequality in \cite[Eq.\ (3.18)]{carlen_upper_1986}	 holds true, then there exists a constant $c'>0$ depending only on the dimension $d\geq 1$ and $\mathfrak{c}$ such that
\begin{equation}\label{eq: conclusion Th.3.25 C_K_Stroock}
	q^{\boldsymbol{\alpha}}_t(x,y)\le \frac{c'}{1\vee \sqrt{t^d}}e^{-D(2t; x, y)}\ ,
\end{equation}
where 
\begin{equation}
D(r;x, y):= \sup_{\psi \in \ell^\infty (\Z^d,\boldsymbol \alpha)}\left(   \abs{ \psi(x)-\psi(y)}-r \Gamma (\psi)^2    \right)
\end{equation}
and 
\begin{equation}\Gamma (\psi)^2:= \sup_{x\in \Z^d} \left\{  \sum_{y:|y-x|=1}  \frac{\alpha_y}{2}   \left( e^{\psi(y)-\psi(x)} -1   \right)^2     \right\}\ ,
	\end{equation}
with the above quantity corresponding to the equation one line above \cite[Theorem 3.9]{carlen_upper_1986}.
For what concerns  Nash inequality, since $\alpha(x)\alpha(y)\ge 1$ for all $x,y\in\Z^d$,  we have, for all $f\in \ell^1(\Z^d, \boldsymbol \alpha)$,
\begin{align}\nonumber\label{eq:pre-nash}
	\mathcal{E}_{\boldsymbol \alpha}(f,f	)&:=\frac{1}{2}\sum_{x \in \Z^d}\sum_{y:|y-x|=1} \alpha(x)\alpha(y)\left(f(y)-f(x)\right)^2\\
	\nonumber
	&\ge \frac{1}{2}\sum_{x \in \Z^d}\sum_{y:|y-x|=1} (f(y)-f(x))^2\\
	&\ge C \norm{f}^{2+\frac{4}{d}}_{\ell^2(\Z^d,\nu)}\norm{f}_{\ell^1(\Z^d,\lambda)}^{-\frac{4}{d}}\ ,
\end{align}
where $\lambda$ is the counting measure on $\Z^d$. Note that  for the last inequality above we used   Nash inequality for the  continuous-time simple	 random walk (see,  e.g., \cite[Eq.\ (1.8)]{stroock_markov_1997}), with the constant $C > 0$ depending only on the dimension $d \geq 1$. Due to the assumed uniform ellipticity of $\boldsymbol \alpha$ (see Assumption \ref{definition  disorder}), the equivalence of the norms  $\left\|\cdot\right\|_{\ell^p(\Z^d,\lambda)}$ and $\left\| \cdot\right\|_{\ell^p(\Z^d,\boldsymbol \alpha)}$ together with \eqref{eq:pre-nash} yield
\begin{equation}
	\mathcal{E}_{\boldsymbol \alpha}(f,f)\ge C\, \mathfrak{c}^{-\left(1+\frac{2}{d}\right)}\norm{f}^{2+\frac{4}{d}}_{\ell^2(\Z^d, \boldsymbol \alpha)}\norm{f}_{\ell^1(\Z^d, \boldsymbol \alpha)}^{-\frac{4}{d}}\ ,
\end{equation} 
which corresponds to \cite[Eq.\ (3.18)]{carlen_upper_1986} with $A=\left(C^{-1} \mathfrak{c}^{1+\frac{2}{d}}\right)$, $\nu=d$ and $\delta=0$. Therefore, we get \eqref{eq: conclusion Th.3.25 C_K_Stroock} by \cite[Theorem 3.25]{carlen_upper_1986} for $\rho=1$.

Finally, by arguing as in the proof of Lemma 1.9 in \cite{stroock_markov_1997} and by the uniform ellipticity of $\boldsymbol \alpha$, we obtain the following proposition:
\begin{proposition}[{Heat kernel upper bound}]\label{proposition:heat_kernel_bounds}
	There exists a constant $c>0$ depending only on $d \geq 1$ and $\mathfrak{c}$ such that, for all environments $\boldsymbol \alpha$, $t > 0$ and $x, y \in \Z^d$, the following upper bound holds:
	\begin{equation}\label{eq:upper bound kernel}
		P^{\boldsymbol \alpha}\left(X^{\boldsymbol \alpha,x}_t = y \right)\le \frac{c}{1\vee \sqrt{t^d}}e^{ -\frac{\abs{x-y}}{1\vee \sqrt{t}}}.
	\end{equation} 	
\end{proposition}

\subsection{Proof of Theorem \ref{theorem:uniform_convergence_semigroups}}\label{section:proof_uniform_convergence_semigroups} Let us conclude the proof of Theorem \ref{theorem:uniform_convergence_semigroups}.
\begin{proof}[Proof of Theorem \ref{theorem:uniform_convergence_semigroups}]
	First we prove that, for all $\boldsymbol \alpha \in \mathfrak B$ (see \eqref{eq:B}),  for all $t\ge0$ and $G \in \mathcal C_0(\R^d)$, we have	
	\begin{align}\label{eq:semigroup_uniform_space}
	\sup_{x \in \Z^d}
	\left|S^{N,\boldsymbol \alpha}_{tN^2}G(\tfrac{x}{N}) - \mathcal S_t^\varSigma G(\tfrac{x}{N})\right|\ \underset{N\to \infty}\longrightarrow\ 0\ .
	\end{align}
We  follow the ideas in \cite[Appendix A.2]{chen_quenched_2015}.  For all $u \in \R^d$ and $\varepsilon > 0$, let $\mathcal B_\varepsilon(u)$ (resp. $\closure{\mathcal B_\varepsilon(u)}$) denote the open (resp. closed) Euclidean ball of radius $\varepsilon > 0$ centered in $u \in \R^d$. Moreover, for all $\boldsymbol \alpha$, we define
	\begin{align*}
	\sigma^N_\varepsilon(u):=\inf\left\{t \geq 0: \frac{X^{\boldsymbol \alpha,0}_{tN^2}}{N}\in \mathcal B_\varepsilon(u) \right\}\quad \text{and}\quad \sigma_\varepsilon(u):=\inf \left\{ t \geq 0: B^\varSigma_t \in \mathcal B_\varepsilon(u)\right\}
	\end{align*}	
	to 	be the first hitting times of $\mathcal B_\varepsilon(u)$ of the random walks and Brownian motion, respectively. Then, as a consequence of Proposition \ref{proposition:QFCLT_site_disorder} (see also Remark \ref{remark:B}) and the strong Markov property of both processes, we have, for all  $\boldsymbol \alpha \in \mathfrak B$, for all $t\geq 0$, $T > 0$ and $G \in \mathcal C_0(\R^d)$, 
	\begin{align}\label{eq:vanish1}
	\sum_{\frac{y}{N}\in\closure{\mathcal B_\varepsilon(u)}\cap \tfrac{\Z^d}{N}} E^{\boldsymbol \alpha}\left[G\left(\frac{X^{\boldsymbol \alpha,y}_{tN^2}}{N}\right) \right] P_{\varepsilon,u,T}^{\boldsymbol \alpha}\left( \frac{y}{N}\right)\underset{N\to \infty}\longrightarrow
	\int_{\closure{\mathcal B_\varepsilon(u)}} \mathtt E\left[G(B^\varSigma_t+v)\right]\, \mathtt P_{\varepsilon,u,T}(\dd v)\ ,
	\end{align}
	where
	\begin{align*}
	P^{\boldsymbol \alpha}_{\varepsilon,u,T}\left(\tfrac{y}{N} \right):=P^{\boldsymbol \alpha}\left(\frac{X^{\boldsymbol \alpha,0}_{\sigma^N_\varepsilon(u)}}{N} = \frac{y}{N}\bigg| \sigma^N_\varepsilon(u)< T \right)
	\quad\text{and}\quad
	\mathtt P_{\varepsilon,u,T}\left(\dd v \right):=\mathtt P(B^\varSigma_{\sigma_\varepsilon(u)}=\dd v\big|\sigma_\varepsilon(u) < T)\ .
	\end{align*}
	Let $\{x_N\}_{N \in \N}\subseteq \Z^d$ be such that $\frac{x_N}{N}\to u$ as $N\to \infty$. Then, by the triangle inequality, we have, for all $\varepsilon > 0$, 
	\begin{align}\nonumber
	&\left|S^{N,\boldsymbol \alpha}_{tN^2}G(\tfrac{x_N}{N}) - \mathcal S^\varSigma_t(\tfrac{x_N}{N}) \right|\\
	\nonumber
	\leq&\ \left|S^{N,\boldsymbol \alpha}_{tN^2}G\left(\tfrac{x_N}{N}\right) -	\sum_{\frac{y}{N}\in\closure{\mathcal B_\varepsilon(u)}\cap \tfrac{\Z^d}{N}} E^{\boldsymbol \alpha}\left[G\left(\frac{X^{\boldsymbol \alpha,y}_{tN^2}}{N}\right) \right] P_{\varepsilon,u,T}^{\boldsymbol \alpha}\left( \frac{y}{N}\right)\right|\\
	\nonumber
	+&\ \left|	\sum_{\frac{y}{N}\in\closure{\mathcal B_\varepsilon(u)}\cap \tfrac{\Z^d}{N}} E^{\boldsymbol \alpha}\left[G\left(\frac{X^{\boldsymbol \alpha,y}_{tN^2}}{N}\right) \right] P_{\varepsilon,u,T}^{\boldsymbol \alpha}\left( \frac{y}{N}\right) - 	\int_{\closure{\mathcal B_\varepsilon(u)}} \mathtt E\left[G(B^\varSigma_t+v)\right]\, \mathtt P_{\varepsilon,u,T}(\dd v)\right|\\
	\label{3}
	+&\ \left|	\int_{\closure{\mathcal B_\varepsilon(u)}} \mathtt E\left[G(B^\varSigma_t+v)\right]\, \mathtt P_{\varepsilon,u,T}(\dd v) - \mathcal S^\varSigma_tG(\tfrac{x_N}{N}) \right|\ .
	\end{align}
As for the first term on the r.h.s.\ above, for all environments $\boldsymbol \alpha \in \mathfrak B$, we have,   by H\"older's inequality, 
	\begin{align*}
	 &\left|S^{N,\boldsymbol \alpha}_{tN^2}G\left(\tfrac{x_N}{N}\right) -	\sum_{\frac{y}{N}\in\closure{\mathcal B_\varepsilon(u)}\cap \tfrac{\Z^d}{N}} E^{\boldsymbol \alpha}\left[G\left(\frac{X^{\boldsymbol \alpha,y}_{tN^2}}{N}\right) \right] P_{\varepsilon,u,T}^{\boldsymbol \alpha}\left( \frac{y}{N}\right)\right|\\
	&\leq	\sum_{\frac{y}{N}\in\closure{\mathcal B_\varepsilon(u)}\cap \tfrac{\Z^d}{N}}\left|S^{N,\boldsymbol \alpha}_{tN^2}G\left(\tfrac{x_N}{N}\right)- S^{N,\boldsymbol \alpha}_{tN^2}G\left(\tfrac{y}{N}\right)\right| P_{\varepsilon,u,T}^{\boldsymbol \alpha}\left( \frac{y}{N}\right)\leq \sup_{\frac{y}{N}\in \closure{\mathcal B_\varepsilon(u)}}\left|S^{N,\boldsymbol \alpha}_{tN^2}G\left(\tfrac{x_N}{N}\right)- S^{N,\boldsymbol \alpha}_{tN^2}G\left(\tfrac{y}{N}\right) \right| \ .
\end{align*}
The above upper bound, since $\tfrac{x_N}{N}\to u$ as $N\to \infty$, yields, by Proposition \ref{prop:holder cont} and the uniform boundedness of the function $G \in \mathcal C_0(\R^d)$,
\begin{equation}\label{eq:vanish2}
\lim_{\varepsilon \to 0}\limsup_{N\to \infty}	\left|S^{N,\boldsymbol \alpha}_{tN^2}G\left(\tfrac{x_N}{N}\right) -	\sum_{\frac{y}{N}\in\closure{\mathcal B_\varepsilon(u)}\cap \tfrac{\Z^d}{N}} E^{\boldsymbol \alpha}\left[G\left(\frac{X^{\boldsymbol \alpha,y}_{tN^2}}{N}\right) \right] P_{\varepsilon,u,T}^{\boldsymbol \alpha}\left( \frac{y}{N}\right)\right|=0	
\end{equation}	
for all environments 	$\boldsymbol \alpha \in \mathfrak B$ and $t \geq 0$. A similar argument employing the uniform continuity of $G \in \mathcal C_0(\R^d)$ and the translation invariance of the law of Brownian motion ensures
\begin{equation}\label{eq:vanish3}
	\lim_{\varepsilon \to 0}\limsup_{N\to \infty}	\left|	\int_{\closure{\mathcal B_\varepsilon(u)}} \mathtt E\left[G(B^\varSigma_t+v)\right]\, \mathtt P_{\varepsilon,u,T}(\dd v) - \mathcal S^\varSigma_tG(\tfrac{x_N}{N})\right|=0
\end{equation}
for all $t \geq 0$. By combining \eqref{eq:vanish1}--\eqref{eq:vanish3}, we obtain, for all  $\boldsymbol \alpha \in \mathfrak B$, for all $G \in \mathcal C_0(\R^d)$, $t \geq 0$, $u \in \R^d$ and approximating points $\frac{x_N}{N}\to u$,
	\begin{equation}\label{eq:pointwise}
	\left| S^{N,\boldsymbol \alpha}_{tN^2}G(\tfrac{x_N}{N}) - \mathcal S^\varSigma_tG(\tfrac{x_N}{N})\right|\ \underset{N\to \infty}\longrightarrow 0\ .
\end{equation}

	In order to go from pointwise (i.e., \eqref{eq:pointwise}) to uniform convergence over points in $\Z^d$ (i.e., \eqref{eq:semigroup_uniform_space}), we  crucially use	 the heat kernel upper bound in Proposition \ref{proposition:heat_kernel_bounds} and the H\"older equicontinuity in Proposition \ref{prop:holder cont}. First, note that proving \eqref{eq:semigroup_uniform_space}  for continuous and compactly supported functions $G \in \mathcal C_c(\R^d)$  suffices, due to the density of $\mathcal C_c(\R^d)$ in $\mathcal C_0(\R^d)$ and the contractivity of the semigroups $\{S^{N,\boldsymbol \alpha}_{tN^2}, t \geq 0\}$ and $\{\mathcal S^\varSigma_t, t \geq 0\}$ w.r.t.\ the supremum norms on $\frac{\Z^d}{N}$ and $\R^d$, respectively. Hence, for all $G\in \mathcal C_c(\R^d)$ and for all compact sets $\mathcal K\subseteq \R^d$, \eqref{eq:pointwise}, Proposition \ref{prop:holder cont} and the uniform continuity of $\mathcal S^\varSigma_t G \in \mathcal C_0(\R^d)$ imply, for all $\boldsymbol \alpha \in \mathfrak B$,
	\begin{equation}
		\sup_{\frac{x}{N}\in \mathcal K\cap \frac{\Z^d}{N}}  \left| S^{N,\boldsymbol \alpha}_{tN^2}G(\tfrac{x}{N}) - \mathcal S^\varSigma_tG(\tfrac{x}{N})\right|\ \underset{N\to \infty}\longrightarrow 0\ .
	\end{equation}
	Letting $\text{\normalfont{supp}}(G) \subseteq \R^d$ denote the compact support of $G \in \mathcal C_c(\R^d)$, we have
	\begin{multline}\label{eq:fuori_compact}
		\sup_{\frac{x}{N}\in \mathcal K^c \cap \frac{\Z^d}{N}} \left| S^{N,\boldsymbol \alpha}_{tN^2}G(\tfrac{x}{N}) - \mathcal S^\varSigma_tG(\tfrac{x}{N})\right|\\
		\leq \sup_{u\in \R^d}\left| G(u)\right| 	\sup_{\frac{x}{N}\in \mathcal K^c\cap \frac{\Z^d}{N}} \left\{P^{\boldsymbol \alpha}\left(\frac{X^{\boldsymbol \alpha,x}_{tN^2}}{N}\in \text{\normalfont supp}(G)\right) + \mathtt P(B^\varSigma_t+x\in \text{\normalfont supp}(G)) \right\}\ .
	\end{multline}
Thus, by the heat kernel upper bounds for $\RWa$ (Proposition \ref{proposition:heat_kernel_bounds}) and analogous bounds for the non-degenerate Brownian motion $\{B_t^\varSigma, t \geq 0\}$,  we can choose $\mathcal K\subseteq \R^d$ such that the r.h.s.\ in \eqref{eq:fuori_compact} is arbitrarily small. This yields \eqref{eq:semigroup_uniform_space} for all $\boldsymbol \alpha \in \mathfrak B$.

To go from \eqref{eq:semigroup_uniform_space} to \eqref{eq:uniform_convergence_semigroups} in which the convergence is uniform over bounded intervals of time, we apply \cite[Chapter 1, Theorem 6.1]{EthierKurtz}. Indeed, for all realizations of the environment $\boldsymbol \alpha$, the semigroups $\{S^{N,\boldsymbol \alpha}_t,\, t\geq 0\}_{N\in \N}$ and $\{\mathcal S^\varSigma_t,\, t \geq 0\}$ are strongly continuous contraction semigroups in the Banach spaces $\{\mathcal C_0(\frac{\Z^d}{N}) \}_{N\in \N}$  and $\mathcal C_0(\R^d)$ (endowed with the corresponding supremum norms), respectively; moreover,  the projections $\pi_N :  \mathcal C_0(\R^d) \rightarrow \mathcal C_0(\tfrac{\Z^d}{N})$ given by $	\pi_NG(\tfrac{x}{N}) := G(\tfrac{x}{N})$ are linear and such that $\sup_{N\in \N}\,\left\|\pi_N\right\|_N=1<\infty$, with $\left\|\pi_N\right\|_N$ denoting the operator norm of $\pi_N$.  \end{proof}

\begin{remark}[Equivalent formulations of the arbitrary starting point quenched invariance principle]\label{remark: equivalence AQPQIP and convergence semigroup}
If one assumes, for a given realization of the environment $\boldsymbol \alpha$, the  invariance principle for the random walk $\RWa$ with arbitrary   starting positions, i.e.,
	\begin{equation}\label{eq:ASPQIP}
	\parbox{\dimexpr\linewidth-4em}{%
		\strut
		For all $T > 0$, for any  macroscopic  point $u\in \R^d$ and for any sequence of points $\{x_N\}_{N \in \N} \subseteq \Z^d$ such that $\tfrac{x_N}{N}\to u$ as $N \to \infty$,  the laws of  $\{\tfrac{1}{N}X^{\boldsymbol \alpha,x_N}_{tN^2},\, t \in [0,T]\}_{N \in \N}$, the diffusively rescaled $\RWa$ started from $\tfrac{x_N}{N}$,  converge weakly to the law of  $\{B^{\varSigma, u}_t:=B^\varSigma_t+u,\ t \in [0,T]\}$, the  Brownian motion started from $u \in \R^d$ and with a non-degenerate covariance matrix $\varSigma$ independent of the  realization of the environment $\boldsymbol \alpha$
		\strut}
	\end{equation}
	then  \eqref{eq:pointwise} follows immediately by the uniform continuity of $\mathcal S^\varSigma_t G \in \mathcal C_0(\R^d)$. By the same argument used in the final part of the proof of Theorem \ref{theorem:uniform_convergence_semigroups} above (i.e.,  the part of the proof immediately after   \eqref{eq:pointwise} involving the heat kernel upper bound in Proposition \ref{proposition:heat_kernel_bounds} and the H\"older equicontinuity in Proposition \ref{prop:holder cont})
	one gets the convergence in \eqref{eq:uniform_convergence_semigroups0}. Therefore, in view of the discussion just after 
	\eqref{eq:uniform_convergence_semigroups0}, we obtain that, under Assumption \ref{definition  disorder}, \eqref{eq:ASPQIP} and \eqref{eq:uniform_convergence_semigroups0} are equivalent.
	
\end{remark}

\begin{remark}[{Quenched local CLT}]
	As already mentioned,  \eqref{eq:uniform_convergence_semigroups}, namely the arbitrary starting point quenched invariance principle for the diffusively rescaled random walks $\RWa$,  is stronger than the quenched invariance principle for $\RWa$ starting from the origin. Another well-known strengthening of the quenched invariance principle  is the \emph{quenched local central limit theorem}  (see, e.g., \cite[Theorem 1.11 and Remark 1.12]{andres_invariance2015}, which applies to our context) for $\RWa$: if we denote by $k^{\varSigma}_t$ the heat kernel of the Brownian motion started at the origin, it holds that, for $\mathcal P$-a.e.\ environment $\boldsymbol \alpha$ and for any $\ell$, $T> 0$ and $\delta>0$,  
	\begin{align}\label{eq: local clt}
	\lim_{N\to \infty} \sup_{\abs{\frac{y}{N}}<\ell}\sup_{t\in[\delta,T]} \left| N^d P^{\boldsymbol \alpha}\left(\frac{X^{\boldsymbol \alpha,0}_{tN^2}}{N} =\frac{y}{N}  \right) - k^{\varSigma}_t(\tfrac{y}{N})  \right|\ =\ 0\ .
	\end{align}
	The proof of \eqref{eq: local clt} resembles that of Theorem \ref{theorem:uniform_convergence_semigroups} and, thus, one may wonder whether \eqref{eq: local clt} directly yields \eqref{eq:uniform_convergence_semigroups}.
	However, \eqref{eq: local clt} does  not seem to be of help when proving \eqref{eq:semigroup_uniform_space},  being the supremum over space in the arrival point and not in the starting point -- fixed to be the origin -- and being the supremum over time only on bounded intervals away from $t=0$.
	
\end{remark}

\section{Proof of the hydrodynamic limit}\label{section proof HDL}

In this section we present the proof of  Theorem \ref{theorem HDL}, which consists of two steps: ensuring tightness of the empirical density fields and establishing convergence of their finite dimensional distributions to the unique solution of \eqref{heat equation}. In both steps, we use the following representation for the renormalized occupation variables: for all realizations of the environment $\boldsymbol \alpha$, there exists a probability space such that a.s.,  for all initial configurations $\eta \in \mathcal X^{\boldsymbol \alpha}$, for all $x \in \Z^d$ and $t \geq 0$, 
\begin{align}\label{eq:mild_solution_representation}
\tfrac{\eta_t(x)}{\alpha_x}= S^{\boldsymbol \alpha}_t(\tfrac{\eta(\cdot)}{\alpha_\cdot})(x) + \int_0^t S^{\boldsymbol \alpha}_{t-s} \dd M^{\boldsymbol \alpha}_s(x)\ ,
\end{align} 
where $\{M^{\boldsymbol \alpha}_\cdot(x),\, x \in \Z^d\}$ is a family of square integrable martingales w.r.t.\ the natural filtration of $\SEPa$ (see also \eqref{mild solution}--\eqref{eq:martingale_occupation_variable}), 	
whose predictable quadratic covariations are given by
\begin{align}\label{eq:predictable_cov1}
\langle M^{\boldsymbol \alpha}(x),M^{\boldsymbol \alpha}(y)\rangle_t = - \ind_{\{|x-y|=1\}}\int_0^t \alpha_x \alpha_y\, \left(\tfrac{\eta_s(x)}{\alpha_x} - \tfrac{\eta_s(y)}{\alpha_y} \right)^2\, \dd s	
\end{align}	
for $x, y \in \Z^d$ with $x \neq y$, 
and
\begin{align}\label{eq:predictable_cov2}
\langle M^{\boldsymbol \alpha}(x),M^{\boldsymbol \alpha}(x)\rangle_t= - \sum_{\substack{y\in \Z^d\\|y-x|=1}} \langle M^{\boldsymbol \alpha}(x),M^{\boldsymbol \alpha}(y)\rangle_t
\end{align}
for $x \in \Z^d$.
The identity in \eqref{eq:mild_solution_representation} expresses the solution of the following infinite system of stochastic differential equations (cf.\ \eqref{equation duality function}--\eqref{equation duality})
\begin{align*}
\left\{
\begin{array}{rcll}
\dd (\tfrac{\eta_t(\cdot)}{\alpha_\cdot})(x) &=&  A^{\boldsymbol \alpha}(\tfrac{\eta_{t^-}(\cdot)}{\alpha_\cdot})(x)\, \dd t + \dd M^{\boldsymbol \alpha}_t(x)\  , &x\in \Z^d\ ,\ 	t\geq 0\\[.15cm]
\tfrac{\eta_0(x)}{\alpha_x} &=& \tfrac{\eta(x)}{\alpha_x}\ , &x \in \Z^d\ ,
\end{array}
\right.
\end{align*}
as a mild solution (see, e.g., \cite[Chapter 6, Section 1]{prato_stochastic_2014}). 	The rigorous proof of the identity in \eqref{eq:mild_solution_representation}  -- in which the r.h.s.\ contains infinite summations -- is postponed to Appendix \ref{section:ladder} below. The idea of the proof is to first provide a so-called \textquotedblleft ladder representation\textquotedblright\ for $\SEPa$ in terms of a symmetric exclusion process which allows at most one particle per site; then obtain a mild solution representation analogous to the one in \eqref{eq:mild_solution_representation} for such \textquotedblleft ladder\textquotedblright\ exclusion process as done in, e.g., \cite{nagy_symmetric_2002, faggionato_bulk_2007, redig_symmetric_2020}. The same strategy can be applied to rigorously verify the identities in \eqref{eq:predictable_cov1}--\eqref{eq:predictable_cov2}. We refer to Appendix \ref{section:ladder} for further details.

\subsection{Tightness} \label{section tightness}
In the proof of tightness for the empirical density fields we employ the uniform convergence over time  of the semigroups established in Theorem \ref{theorem:uniform_convergence_semigroups} and Corollary \ref{corollary:l1}. Tightness in quenched random environment, which by Mitoma's tightness criterion \cite{mitoma_tightness_1983} follows from tightness of the following real-valued processes
\begin{align}\label{eq:sequence_fields_2}
\left\{\mathsf X^N_t(G),\, t \in [0,T]	  \right\}_{N \in \N}\ ,\qquad \forall\, G \in \mathscr S(\R^d)\ ,
\end{align}
has been established via the  strategy of employing corrected empirical density fields (\cite{jara_quenched_2008, goncalves_scaling_2008, faggionato_hydrodynamic_2009, faggionato_zerorange_2010} and \cite{jara_hydrodynamic_2011}). In what follows, we opt for a different strategy by applying    the tightness criterion developed in \cite[Appendix B]{redig_symmetric_2020}, which, for convenience of the reader, we report below.
\begin{theorem}[{Tightness criterion} {\cite[Theorem B.4]{redig_symmetric_2020}}]\label{theorem:tightness_criterion}
	For a fixed $T > 0$, let	 $\left\{\mathsf Z^N_t,\, t \in [0,T]\right\}_{N \in \N}$ be a family of real-valued stochastic processes with laws $\{\mathscr P^N\}_{N\in \N}$. Then,  this family is tight in the Skorokhod space $\mathcal D([0,T],\R)$ if the following  conditions hold:
	\begin{enumerate}[label={\normalfont (T\arabic*)},ref={\normalfont (T\arabic*)}]
		\item \label{it:T1}For all $t$ in a dense subset of $[0,T]$ which includes $T$, 
		\begin{align*}
		\lim_{\ell\to \infty} \limsup_{N\to \infty} \mathscr P^N\left(\left|\mathtt Z^N_t \right| > \ell \right)=0\ .
		\end{align*}
		\item \label{it:T2}For all $\varepsilon > 0$, there exists $h_\varepsilon > 0$ and $N_\varepsilon \in \N$ such that, for all $N \geq N_\varepsilon$, there exist deterministic functions $\psi^N_\varepsilon, \psi_\varepsilon : [0,h_\varepsilon]\to [0,\infty)$ and non-negative values $\phi^N_\varepsilon$ satisfying the following properties:
		\begin{enumerate}[label={\normalfont (\roman*)}]
			\item \label{it:i}The functions $\psi^N_\varepsilon$ are non-decreasing.
			\item \label{it:ii} For all $h \in [0,h_\varepsilon]$ and $t \in [0,T]$, we have
			\begin{align*}
			\mathscr P^N\left(\left|\mathtt Z^N_{t+h}-\mathtt Z^N_t \right|>\varepsilon\big| \mathscr F^N_t \right)\leq \psi^N_\varepsilon(h)\ ,\quad \text{\normalfont a.s.}\ ,
			\end{align*}
			where $\left\{\mathscr F^N_t, t \geq 0 \right\}$ denotes the natural filtration of $\left\{\mathtt Z^N_t, t \geq 0 \right\}$.
			\item \label{it:iii} For all $h \in [0,h_\varepsilon]$, we have $\psi^N_\varepsilon(h)\leq \psi_\varepsilon(h)+\phi^N_\varepsilon$.
			\item \label{it:iv} $\phi^N_\varepsilon\to 0$ as $N\to \infty$.
			\item \label{it:v} $\psi_\varepsilon(h)\to 0$ as $h\to 0$.
		\end{enumerate}
	\end{enumerate}
	
\end{theorem}

As we show in the proof of Proposition \ref{prop:fortightness} below, this criterion, the semigroup convergence in Theorem \ref{theorem:uniform_convergence_semigroups} and  the following mild solution representation of the empirical density fields (see also \eqref{elll})
\begin{equation}\label{eq:second_mild_solution}
\mathsf X^N_{t+h}(G)= \mathsf X^N_t(S^{N,\boldsymbol \alpha}_{hN^2}G) + \int_{tN^2}^{(t+h)N^2} \dd\mathsf M^N_s(S^{N,\boldsymbol \alpha}_{(t+h)N^2-s}G)\ ,\quad t, h > 0\ ,\ G \in \mathscr S(\R^d)\ ,
\end{equation}
  yield  tightness directly for the processes in \eqref{eq:sequence_fields_2}.

\begin{proposition}[{Tightness}]\label{prop:fortightness} For all environments $\boldsymbol \alpha \in \mathfrak A \cap \mathfrak B$ (see \eqref{eq:A} and \eqref{eq:B}) and for all $T > 0$, the sequence $$\left\{\mathsf X^N_t,\, t \in [0,T]  \right\}_{N \in \N}$$ is tight in $\mathcal D([0,T],\mathscr S'(\R^d))$. As a consequence, $\{\mathsf X^N_t,\, t \geq 0\}_{N \in \N}$ is tight in $\mathcal D([0,\infty),\mathscr S'(\R^d))$.
\end{proposition}
\begin{proof}
All throughout the proof, we fix $\boldsymbol \alpha \in \mathfrak A\cap \mathfrak B$.
	As mentioned above, it suffices to  show that conditions \ref{it:T1} and \ref{it:T2} in Theorem \ref{theorem:tightness_criterion}   hold for 
	\begin{equation}\label{eq:processes}
	\left\{\mathsf Z^N_t,\, t\in [0,T]  \right\}_{N \in \N}=\left\{\mathsf X^N_t(G),\, t\in [0,T] \right\}_{N \in \N}\ ,
	\end{equation}
	for all $G \in \mathscr S(\R^d)$. Because \ref{it:T1} is a consequence of Proposition  \ref{proposition finite-dim distributions} below, it suffices to show \ref{it:T2}. To this purpose, we fix $G \in \mathscr S(\R^d)$ and set, for all $\varepsilon > 0$, $h \geq 0$ and $N \in \N$, 	
	\begin{align}\label{eq:psiNeps}
	\psi^N_\varepsilon(h)\coloneqq&\ \frac{C}{\varepsilon^2} \sup_{h' \in [0,h]} \sup_{x \in \Z^d}\left|G(\tfrac{x}{N}) - S^{N,\boldsymbol \alpha}_{h'N^2}G(\tfrac{x}{N}) \right|\\
	\label{eq:psieps}
	\psi_\varepsilon(h) \coloneqq&\ \frac{C}{\varepsilon^2} \sup_{h' \in [0,h]} \sup_{u \in \R^d}\left|G(u) - \mathcal S^\varSigma_{h'}G(u) \right|
	\end{align}
	and
	\begin{equation}\label{eq:phiNeps}
	\phi^N_\varepsilon \coloneqq \frac{C}{\varepsilon^2} \sup_{t \in [0,T]}\sup_{x \in \Z^d}\left|S^{N,\boldsymbol \alpha}_{tN^2}G(\tfrac{x}{N})-\mathcal S^\varSigma_tG(\tfrac{x}{N}) \right|\ ,
	\end{equation}
	where	
	\begin{equation}
		C\coloneqq\sup_{N\in \N} \frac{1}{N^d}\sum_{x \in \Z^d}\left|G(\tfrac{x}{N}) \right|\alpha_x 
	\end{equation}
is a  constant independent of $N\in \N$ and,  since $\boldsymbol \alpha \in \mathfrak A$ (see \eqref{eq:A}), finite. As a consequence of the triangle inequality, Theorem \ref{theorem:uniform_convergence_semigroups} and the continuity of $h\in [0,\infty)\mapsto \mathcal S^\varSigma_h G\in \mathcal C_0(\R^d)$,  the functions in \eqref{eq:psiNeps}--\eqref{eq:phiNeps} satisfy the conditions in items \ref{it:i}, \ref{it:ii}, \ref{it:iv} and \ref{it:v} of the tightness criterion in Theorem \ref{theorem:tightness_criterion}. In the remainder of the proof, we verify also the remaining condition \ref{it:iii} in that theorem.

By \eqref{eq:second_mild_solution} and the triangle inequality, we have, for all $t, h \geq 0$ and $N \in \N$, 
\begin{align}\label{eq:bound}
\P^{\boldsymbol \alpha}_{\nu_N^{\boldsymbol \alpha}}\left(\left|\mathsf X^N_{t+h}(G)-\mathsf X^N_t(G) \right|>\varepsilon \big| \mathcal  F^N_t \right)
	\leq&\ \P^{\boldsymbol \alpha}_{\nu_N^{\boldsymbol \alpha}}\left(\left|\mathsf X^N_t(S^{N,\boldsymbol \alpha}_{hN^2}G-G) \right|>\frac{\varepsilon}{2}\Big|\mathcal  F^N_t \right)\\
	\nonumber
+&\ \P^{\boldsymbol \alpha}_{\nu_N^{\boldsymbol \alpha}}\left(\left| \int_{tN^2}^{(t+h)N^2} \dd\mathsf M^N_s(S^{N,\boldsymbol \alpha}_{(t+h)N^2-s}G) \right|>\frac{\varepsilon}{2}\Bigg\vert\mathcal  F^N_t \right)\ ,
\end{align}
with $\mathcal  F^N_t\coloneqq\sigma\{\mathsf X^N_s,\, s \leq t\}$.
	The boundedness of the occupation variables of $\SEPa$,  the convergence in \eqref{eq:l1sup_convergence} and the continuity of $h \in [0,\infty)\mapsto \mathcal S^\varSigma_h G \in \mathcal C_0(\R^d)$ allows us to choose $h_\varepsilon > 0$ and $N_\varepsilon \in \N$ such that the first term on the r.h.s.\ in \eqref{eq:bound} equals zero for all $h\in [0,h_\varepsilon]$, $N \geq N_\varepsilon$ and $t \geq 0$, i.e., 
	\begin{equation}\label{eq:bound2}
		\P^{\boldsymbol \alpha}_{\nu_N^{\boldsymbol \alpha}}\left(\left|\mathsf X^N_t(S^{N,\boldsymbol \alpha}_{hN^2}G-G) \right|>\frac{\varepsilon}{2}\Big|\mathcal  F^N_t \right) = 0 \ ,\qquad h \in [0,h_\varepsilon]\ ,\ N \geq N_\varepsilon\ .
	\end{equation} As for the second term on the r.h.s.\ in \eqref{eq:bound}, by 	Chebyshev's inequality and the first inequality in \eqref{variance} below, we obtain, for all $h \in [0,h_\varepsilon]$, $N\geq N_\varepsilon$ and $t\geq 0$,
	\begin{equation}\label{eq:bound3}
		 \P^{\boldsymbol \alpha}_{\nu_N^{\boldsymbol \alpha}}\left(\left| \int_{tN^2}^{(t+h)N^2} \dd\mathsf M^N_s(S^{N,\boldsymbol \alpha}_{(t+h)N^2-s}G) \right|>\frac{\varepsilon}{2}\Bigg\vert\mathcal  F^N_t \right)\leq \psi^N_\varepsilon(h)\ ,\qquad \text{a.s.}\ .
	\end{equation}
By combining \eqref{eq:bound}--\eqref{eq:bound3}, condition \ref{it:iii} in Theorem \ref{theorem:tightness_criterion}  holds true for the process  \eqref{eq:processes}, thus yielding the desired result.
\end{proof}

\subsection{Convergence of finite dimensional distributions} \label{section finite dimensional distributions}
In the following proposition -- which is an adaptation of, e.g., \cite[Lemma 12]{nagy_symmetric_2002}, \cite[Lemma 3.1]{faggionato_bulk_2007},  \cite[Lemma 5.1]{redig_symmetric_2020} -- we prove \eqref{equation variance va a 0}. To this purpose, recall the definitions of $\P^{\boldsymbol \alpha}_{\nu_N^{\boldsymbol \alpha}}$, $\P^{\boldsymbol \alpha}_\eta$, $\E^{\boldsymbol \alpha}_{\nu_N^{\boldsymbol \alpha}}$ and $\E^{\boldsymbol \alpha}_\eta$ at the beginning of Section \ref{Section HDL} (below \eqref{equation empirical density field}).
\begin{lemma}\label{proposition variance} 
	For any given realization of the environment $\boldsymbol \alpha$, for all $N \in \N$, $G \in \mathscr S(\R^d)$, $\eta \in \Xa$ and $t \geq 0$, we have 
	\begin{multline}\label{variance}
	\E^{\boldsymbol \alpha}_\eta\left[\left(   \frac{1}{N^d}\sum_{x \in \Z^d} \alpha_x \int_0^{tN^2}  S^{N,\boldsymbol \alpha}_{tN^2-s}G(\tfrac{x}{N})\, \dd M^{\boldsymbol \alpha}_s(x)     \right)^2\right]\\ \leq \frac{1}{2N^d} \frac{1}{N^d}\sum_{x \in \Z^d} \left(\left|G(\tfrac{x}{N}) \right|^2 - \left|S^{N,\boldsymbol \alpha}_{tN^2}G(\tfrac{x}{N}) \right|^2\right) \alpha_x \leq \frac{1}{2N^d} \frac{1}{N^d}\sum_{x \in \Z^d} \left|G(\tfrac{x}{N}) \right|^2 \alpha_x \ .
	\end{multline}
	As a consequence of \eqref{variance} and the uniformity of the upper bound w.r.t.\ $\eta \in \Xa$, we further get
	\begin{align}\label{e1}
	\E^{\boldsymbol \alpha}_{\nu_N^{\boldsymbol \alpha}}\left[\left(   \frac{1}{N^d}\sum_{x \in \Z^d} \alpha_x \int_0^{tN^2}  S^{N,\boldsymbol \alpha}_{tN^2-s}G(\tfrac{x}{N})\, \dd M^{\boldsymbol \alpha}_s(x)     \right)^2\right]\ \underset{N \to \infty}\longrightarrow\ 0\ ,
	\end{align}
	where $\{\nu_N^{\boldsymbol \alpha}\}_{N\in \N}$ is the sequence of probability measures on $\mathcal X^{\boldsymbol \alpha}$ given in Theorem \ref{theorem HDL}.
\end{lemma}
\begin{proof}
	A simple computation employing the explicit form of the predictable quadratic covariations of the martingales \eqref{eq:predictable_cov1}--\eqref{eq:predictable_cov2} yields
	\begin{align*}
	&\E^{\boldsymbol \alpha}_\eta\left[\left(   \frac{1}{N^d}\sum_{x \in \Z^d} \alpha_x \int_0^{tN^2}  S^{N,\boldsymbol \alpha}_{tN^2-s}G(\tfrac{x}{N})\, \dd M^{\boldsymbol \alpha}_s(x)     \right)^2\right]\\
	& = \int_0^{tN^2}\frac{1}{N^{2d}} \sum_{\substack{x, y \in \Z^d\\
			|x-y|=1}} \left(S^{N,\boldsymbol \alpha}_{tN^2-s}G(\tfrac{x}{N})-S^{N,\boldsymbol \alpha}_{tN^2-s}G(\tfrac{y}{N}) \right)^2 \alpha_x\alpha_y\, \E^{\boldsymbol \alpha}_\eta\left[\left(\tfrac{\eta_s(x)}{\alpha_x}-\tfrac{\eta_s(y)}{\alpha_y} \right)^2 \right]\dd s\ .		\end{align*}
	Because a.s.\ $0 \leq \left(\tfrac{\eta_s(x)}{\alpha_x}-\tfrac{\eta_s(y)}{\alpha_y} \right)^2\leq 1$, we further get
	\begin{align*}
	&\E^{\boldsymbol \alpha}_\eta\left[\left(   \frac{1}{N^d}\sum_{x \in \Z^d} \alpha_x \int_0^{tN^2}  S^{N,\boldsymbol \alpha}_{tN^2-s}G(\tfrac{x}{N})\, \dd M^{\boldsymbol \alpha}_s(x)     \right)^2\right]\\
	&\leq \frac{1}{N^d} \int_0^{tN^2} \frac{1}{N^d} \sum_{\substack{x,y\in \Z^d\\|x-y|=1}} \alpha_x\alpha_y  \left(S^{N,\boldsymbol \alpha}_{tN^2-s}G(\tfrac{x}{N})-S^{N,\boldsymbol \alpha}_{tN^2-s}G(\tfrac{y}{N}) \right)^2 \dd s\\
	&=\frac{1}{2N^d} \int_0^{tN^2} -\frac{\dd}{\dd s}\left(\frac{1}{N^d} \sum_{x\in \Z^d} \left|S^{N,\boldsymbol \alpha}_{tN^2-s} G(\tfrac{x}{N}) \right|^2\alpha_x\right)\dd s\\
	&= \frac{1}{2N^d} \left(\frac{1}{N^d}\sum_{x \in \Z^d} \left(\left|G(\tfrac{x}{N}) \right|^2 - \left|S^{N,\boldsymbol \alpha}_{tN^2}G(\tfrac{x}{N}) \right|^2\right) \alpha_x \right)\leq \frac{1}{2N^d} \left(\frac{1}{N^d}\sum_{x \in \Z^d} \left|G(\tfrac{x}{N}) \right|^2 \alpha_x \right)\ .
	\end{align*}
	In view of Lemma \ref{lemma:ergodic_theorem}, $\limsup_{N\to\infty} \frac{1}{N^d}\sum_{x \in \Z^d} \left|G(\tfrac{x}{N}) \right|^2 \alpha_x<\infty$, thus, concluding the proof.
\end{proof}

 Since, with probability one, only one particle jumps at the time,  for all environments $\boldsymbol \alpha$, and for all $T > 0$ and $G \in \mathscr S(\R^d)$, we have
\begin{equation}
	\E^{\boldsymbol \alpha}_{\nu_N^{\boldsymbol \alpha}}\left[	\sup_{t\in [0,T]} \left|\mathsf X^N_{t^+}(G)-\mathsf X^N_{t^-}(G) \right|\right]\leq \frac{2\sup_{u\in \R^d}\left|G(u)\right|}{N^d}\underset{N\to\infty}\longrightarrow 0\ .
\end{equation}
By combining this with the relative compactness of $\{\mathsf X^N_t,\, t \geq 0\}$ in $\mathcal D([0,T],\mathscr S'(\R^d))$ (see Proposition \ref{prop:fortightness}), we obtain that 
 all limit points of $\{\mathsf X^N_t,\, t \geq 0\}$ belong to $\mathcal C([0,T],\, \mathscr S'(\R^d))$.  Hence, Remark \ref{remark:uniquness} and the following proposition conclude the proof of Theorem \ref{theorem HDL}.

\begin{proposition}\label{proposition finite-dim distributions}
	Recall the definitions \eqref{eq:A}, \eqref{eq:B} and \eqref{eq:C}, and fix $\boldsymbol \alpha \in \mathfrak A\cap \mathfrak B \cap \mathfrak C$. Then,  for all $\delta > 0$,  $t \geq 0$ and $G \in \mathscr S(\R^d)$, we have
	\begin{equation}
	\P^{\boldsymbol \alpha}_{\nu_N^{\boldsymbol \alpha}}\left(\left|\mathsf X^N_t(G) - \pi^\varSigma_t(G) \right| > \delta	  \right) \underset{N \to \infty}\longrightarrow 0\ ,
	\end{equation}
	where $\{\pi^\varSigma_t,\, t \geq 0\}$ is given in \eqref{eq:pi_t}.
\end{proposition}
\begin{proof} Due to the uniform boundedness of the environments $\boldsymbol \alpha$ (Assumption \ref{definition  disorder}) and the decomposition \eqref{elll} of the empirical density fields, we obtain, for all $\delta > 0$,
	\begin{multline}\label{r1}
	\P^{\boldsymbol \alpha}_{\nu_N^{\boldsymbol \alpha}}\left(\left|\mathsf X^N_t(G) - \pi^\varSigma_t(G) \right| > \delta\,  \right) \\
	\leq \P^{\boldsymbol \alpha}_{\nu_N^{\boldsymbol \alpha}}\left(\left|\mathsf X^N_0(S^{N,\boldsymbol \alpha}_{tN^2}G) - \pi^\varSigma_t(G)\right| > \frac{\delta}{2}\,  \right)+ \P^{\boldsymbol \alpha}_{\nu_N^{\boldsymbol \alpha}}\left(\left|\int_0^{tN^2} \dd\mathsf M^N_s(S^{N,\boldsymbol \alpha}_{tN^2-s}G)\right| > \frac{\delta}{2} \right)\ .
	\end{multline}
	Hence, by Chebychev's inequality and Lemma \ref{proposition variance}, the second term on the r.h.s.\ in \eqref{r1} vanishes as $N \to \infty$. Concerning the first term on the r.h.s.\ in \eqref{r1}, in view of $\pi^\varSigma_t(G)=\pi^{\bar\rho}(\mathcal S^\varSigma_t G)$ (see Remark \ref{remark:uniquness}), we proceed as follows:
	\begin{align}	\label{r2}\nonumber	
	\P^{\boldsymbol \alpha}_{\nu_N^{\boldsymbol \alpha}}\left(\left|\mathsf X^N_0(S^{N,\boldsymbol \alpha}_{tN^2}G) - \pi^\varSigma_t(G) \right|>\frac{\delta}{2}\right)
	&\leq \P^{\boldsymbol \alpha}_{\nu_N^{\boldsymbol \alpha}}\left(\left| \mathsf X^N_0(S^{N,\boldsymbol \alpha}_{tN^2}G) - \mathsf X^N_0(\mathcal S^\varSigma_tG)\right|>\frac{\delta}{4}\right)
\\
	&+ \P^{\boldsymbol \alpha}_{\nu_N^{\boldsymbol \alpha}}\left(\left|\mathsf X^N_0(\mathcal S^\varSigma_tG) - \pi^{\bar\rho}(\mathcal S^\varSigma_t G) \right|>\frac{\delta}{4}\right)\ .
	\end{align}
	For the first term on the r.h.s.\ in \eqref{r2}, by Markov's inequality and the uniform boundedness of the occupation variables $\{\eta(x),\, x \in \Z^d \}$, we obtain
	\begin{align*}
	\P^{\boldsymbol \alpha}_{\nu_N^{\boldsymbol \alpha}}\left(\left|\mathsf X^N_0(S^{N,\boldsymbol \alpha}_{tN^2}G) - \mathsf X^N_0(\mathcal S^\varSigma_tG) \right| > \frac{\delta}{4} \right) 
	&\leq \frac{4}{\delta}\frac{1}{N^d} \sum_{x \in \Z^d}  \left|S^{N,\boldsymbol \alpha}_{tN^2}G(\tfrac{x}{N})-\mathcal S^\varSigma_tG(\tfrac{x}{N})\right|\alpha_x\ .
	\end{align*}
	In turn, this latter upper bound vanishes for all environments $\boldsymbol \alpha \in \mathfrak A\cap \mathfrak B$, for all $G \in \mathscr S(\R^d)$ and $t \geq 0$,  in view of Corollary \ref{corollary:l1}.
		The second term on the r.h.s.\ in \eqref{r2} vanishes because $\mathscr S(\R^d)$ is invariant under the action of  the Brownian motion semigroup and because of the assumed consistency of the initial conditions (see Definition \ref{definition consistency }) for $\boldsymbol \alpha \in \mathfrak C$ (see \eqref{eq:C}).  
	This concludes the proof. \end{proof}

\begin{appendices}
	
\section{Mild solution  and ladder construction}\label{section:ladder}

In this section we derive the mild solution representation for $\SEPa$. More in detail, we start from a so-called $\boldsymbol \alpha$-ladder symmetric exclusion process (see, e.g., \cite{giardina_duality_2009}), we obtain  the a.s.\ mild solution representation as in e.g. \cite[Section 3]{faggionato_bulk_2007} and \cite[Proposition 4.1]{redig_symmetric_2020} for this ladder counterpart and, then, by means of a projection which preserves the Markov property, we derive an a.s.\  mild solution representation for $\SEPa$.

Let us fix a realization of the environment $\boldsymbol \alpha$  satisfying Assumption \ref{definition  disorder}. Then, we define 
\begin{align}\label{poisson}
	\{\bar{\mathcal N}_\cdot(\{(x,i),(y,j) \}): x, y \in \Z^d\ \text{with}\ |x-y|=1,\, i \in \{1,\ldots, \alpha_x\},\, j \in \{1,\ldots, \alpha_y\} \}.
\end{align}	
to be a family of independent and identically distributed compensated Poisson processes with intensity one.	 

We denote by $(\bar{\mathsf N}, \mathsf F, \{\mathsf F_t: t \geq 0 \}, \mathsf P)$ the probability space on which this compensated Poisson processes are defined.
This randomness will be responsible (see Lemma \ref{lemma mild solution ladder} below) for the stirring construction (see, e.g., \cite[p.\ 399]{liggett_interacting_2005-1}) of the so-called \emph{ladder  symmetric exclusion process with parameter $\boldsymbol \alpha \in  \{1,\ldots,\mathfrak c\}^{\Z^d}$}, the particle system
with configuration space 
\begin{equation}\label{ladder space}
\tilde \Xa = \{\tilde \eta: \tilde \eta(x,i) \in \{0,1 \}\ \text{for all}\ x \in \Z^d\ \text{and}\ i \in \{1,\ldots,\alpha_x\}\}
\end{equation} 
and with infinitesimal generator $\tilde L^{\boldsymbol \alpha}$ acting on  bounded cylindrical functions $\tilde \varphi : \tilde \Xa \to \R$ as follows:
\begin{equation} \label{equation generator ladder}
\tilde L^{\boldsymbol \alpha} \tilde \varphi(\tilde \eta)\ =\  \sum_{\substack{\{x,y\}\in \Z^d\\|x-y|=1}}  \tilde L^{\boldsymbol \alpha}_{xy} \tilde \varphi(\tilde \eta)\ ,
\end{equation}
where
\begin{align*}
\tilde L^{\boldsymbol \alpha}_{xy}\tilde \varphi(\tilde \eta)= \sum_{i=1}^{\alpha_x}\sum_{j=1}^{\alpha_y} &\left\{\tilde \eta(x,i)\, (1-\tilde \eta(y,j))\, (\tilde \varphi(\tilde \eta^{(x,i),(y,j)})- \tilde \varphi(\tilde \eta)) \right.\\
&\left. +\ \tilde \eta(y,j)\, (1-\tilde \eta(x,i))\, (\tilde \varphi(\tilde \eta^{(y,j),(x,i)})-\tilde \varphi(\tilde \eta)) \right\}\ .
\end{align*}
Here $\tilde \eta^{(x,i), (y,j)}$ denotes, also in this context, the configuration obtained from $\tilde \eta\in \tilde {\mathcal X}^{\boldsymbol\alpha}$ by removing a particle at position $(x,i)$ and placing it on $(y,j)$.

This process may be considered as a special case of a  symmetric exclusion process  on the set $\tilde \Z^d = \{(x,i),\, x \in \Z^d,\, i \in \{1,\ldots,\alpha_x\}\}$. For this reason and from the uniform boundedness assumption of the environment, we obtain the following representation of $\{\tilde \eta_t,\,	 t \geq 0 \}$, whose proof is completely analogous to the one  of, e.g., \cite[Section 3]{faggionato_bulk_2007}  and \cite[Proposition 4.3]{redig_symmetric_2020}. We  restate this result below for convenience of the reader.

\begin{lemma}[{Mild solution for the ladder exclusion}]\label{lemma mild solution ladder} Fix an environment $\boldsymbol \alpha \in \{1,\ldots, \mathfrak c\}^{\Z^d}$.
	For $\mathsf P$-a.e.\ realization of the compensated Poisson processes $\{\bar{\mathcal N}_\cdot(\{\cdot,\cdot \}) \}$ and for all initial configurations $\tilde \eta \in \tilde {\mathcal X}_\alpha $, we have, for all $(x,i) \in \tilde \Z^d$ and $t \geq 0$,
	\begin{equation}\label{mild solution ladder}
	\tilde \eta_t(x,i) = \tilde S_t^{\boldsymbol \alpha} \tilde \eta_0(x,i) + \int_0^t \tilde  S_{t-s}^{\boldsymbol \alpha}\, \dd\tilde M^{\boldsymbol \alpha}_s(x,i)\ .
	\end{equation}
In the above formula, $\{\tilde S_t^{\boldsymbol \alpha},\ t \geq 0\}$,  resp.\ $\{\tilde p_t^{\boldsymbol \alpha}(\cdot,\cdot),\ t \geq 0\}$, corresponds to the transition semigroup, resp.\ probabilities, associated to the continuous-time random walk on $\tilde \Z^d$ whose infinitesimal generator $\tilde A^{\boldsymbol \alpha}$ is given below:
\begin{equation*}
	\tilde A^{\boldsymbol \alpha} f(x,i)\ =\ \sum_{\substack{y \in \Z^d\\|y-x|=1}} \sum_{j=1}^{\alpha_y} (f(y,j)-f(x,i))\ ,\qquad (x,i) \in \tilde \Z^d\ ,
\end{equation*}	
where $f:\tilde \Z^d\ \to \R$ is a bounded function. Moreover,  for all $(x,i) \in \tilde \Z^d$ and $t, s \geq 0$, 
	\begin{equation}\label{eq:martingales_tilde}
	\dd \tilde M^{\boldsymbol \alpha}_s(x,i)\equiv \dd\tilde M^{\boldsymbol \alpha}_s((x,i), \tilde \eta_{s^-}) \coloneqq \sum_{\substack{y \in \Z^d\\|y-x|=1}} \sum_{j=1}^{\alpha_y}(\tilde \eta_{s^-}(y,j)-\tilde \eta_{s^-}(x,i))\, \dd\bar{\mathcal N}_s(\{(x,i), (y,j)\})\ ,
	\end{equation}
and
\begin{equation*}
	\int_0^t \tilde  S_{t-s}^{\boldsymbol \alpha}\, \dd\tilde M^{\boldsymbol \alpha}_s(x,i)\coloneqq		\sum_{y \in \Z^d} \sum_{j=1}^{\alpha_y}  \int_0^t  \tilde p_{t-s}^{\boldsymbol \alpha}((x,i),(y,j))\, \dd\tilde M^{\boldsymbol \alpha}_s(y,j)\ ,
\end{equation*}
where the above time-integrals are Lebesgue-Stieltjes integrals w.r.t.\ the realizations of the compensated	 Poisson  	processes.
Furthermore, the infinite summations in \eqref{mild solution ladder} are $\mathsf P$-a.s.\ -- for all times and initial configurations -- absolutely convergent.
\end{lemma}

We leave to the reader to check  that, $\mathsf P$-a.s., for all times $t \geq 0$ and initial configurations $\tilde \eta \in \tilde \Xa$, 
 the predictable quadratic covariations of the martingales $\{\tilde M_t^{\boldsymbol \alpha}(\cdot),\, t \geq 0\}$ in \eqref{mild solution ladder} read as 
\begin{equation}\label{eq:predictable_cov_tilde}
\langle \tilde M^{\boldsymbol \alpha}(x,i), \tilde M^{\boldsymbol \alpha}(y,j)\rangle_t=	- \ind_{\{\left|x-y\right|=1\}}\int_0^t  \left(\tilde \eta_s(x,i)-\tilde \eta_s(y,j) \right)^2\dd s
\end{equation}
for $(x,i), (y,j) \in \tilde \Z^d$ with $x\neq y$, 
and
\begin{equation}\label{eq:predictable_cov_tilde2}
	\langle \tilde M^{\boldsymbol \alpha}(x,i), \tilde M^{\boldsymbol \alpha}(x,i)\rangle_t = -\sum_{\substack{y\in \Z^d\\\left|x-y\right|=1}}\sum_{j=1}^{\alpha_y}\, \langle \tilde M^{\boldsymbol \alpha}(x,i), \tilde M^{\boldsymbol \alpha}(y,j)\rangle_t
\end{equation}
for $(x,i)\in \tilde \Z^d$.

In the following lemma, we show how to obtain  $\SEPa$, with generator given in \eqref{equation generator inhomog SEP},  from the ladder symmetric exclusion process with parameter $\boldsymbol \alpha$ (see, e.g., \cite{giardina_duality_2009} for further details on this construction). By combining this result with Lemma \ref{lemma mild solution ladder}, we obtain a mild solution representation of $\SEPa$ which employs the same randomness used to define the ladder process. 
\begin{lemma}[{Mild solution  for  $\SEPa$}]\label{lemma mild solution exclusion}Fix an environment $\boldsymbol \alpha \in \{1,\ldots, \mathfrak c\}^{\Z^d}$.
	
	Let  $\eta \in \Xa$ and $\tilde \eta \in \tilde \Xa$ be configurations satisfying the following relation:
	\begin{equation}\label{eq:eta_tilde_eta}
		\eta(x)=\sum_{i=1}^{\alpha_x}\tilde \eta(x,i)\ ,\qquad x \in \Z^d\ .
	\end{equation}
	Let $\{\tilde \eta_t,\, t \geq 0 \}$ be the ladder symmetric exclusion process with parameter $\boldsymbol \alpha$, started from $\tilde \eta \in \tilde \Xa$ presented above and represented as in Lemma \ref{lemma mild solution ladder}. Then, the stochastic process $\{\eta_t,\, t \geq 0\}$ taking values in $\Xa$ defined in terms of $\{\tilde \eta_t,\,	 t \geq 0 \}$ as follows
	\begin{align}\label{equation ladder constructin SEP alpha}
	\eta_t(x)\coloneqq \sum_{i=1}^{\alpha_x} \tilde \eta_t(x,i)\ ,\qquad t \geq 0\ ,\ x \in \Z^d\ ,
	\end{align}
	is a Markov process with infinitesimal generator $L^{\boldsymbol \alpha}$ as given in \eqref{equation generator inhomog SEP} and started from $\eta \in \Xa$.
	
	Moreover, for $\mathsf P$-a.e.\ realization of the compensated Poisson processes in \eqref{poisson} and for all initial configurations $\eta \in \Xa$,  we have  (cf.\ the definition of the semigroup $\{S^{\boldsymbol \alpha}_t,\, t \geq 0\}$ in Section \ref{section strategy proof}, as well as \eqref{mild solution}--\eqref{eq:martingale_occupation_variable})
	\begin{equation}\label{mild solution exclusion}
	\left(\tfrac{\eta_t}{\alpha}\right)(x) = S_t^{\boldsymbol \alpha}(\tfrac{\eta}{\alpha})(x)\, +\, \int_0^t S_{t-s}^{\boldsymbol \alpha}\, \dd M^{\boldsymbol \alpha}_s(x)
		\ ,\qquad t \geq0\ ,\ x \in \Z^d\ ,
	\end{equation}
	where  
	\begin{equation}	\label{mart}
	\dd M^{\boldsymbol \alpha}_s(x) \coloneqq \frac{1}{\alpha_x}\sum_{i=1}^{\alpha_x} \dd\tilde M^{\boldsymbol \alpha}_s((x,i))\ ,\qquad x \in \Z^d\ ,
	\end{equation}	
	with $\{\tilde M^{\boldsymbol \alpha}_t(\cdot),\, t \geq 0\}$ being the martingales given in \eqref{eq:martingales_tilde} and defined in terms of the ladder exclusion process $\{\tilde \eta_t,\, t \geq 0\}$ started from any configuration $\tilde \eta \in \tilde \Xa$ related to $\eta \in \Xa$ as in \eqref{eq:eta_tilde_eta}; furthermore, the predictable quadratic covariations of the  martingales in \eqref{mart}  are those given in \eqref{eq:predictable_cov1}--\eqref{eq:predictable_cov2}.
	
\end{lemma}
\begin{proof}
Arguing as in \cite[Theorem 4.2(a)]{giardina_duality_2009}, the process $\{\eta_t,\, t \geq 0\}$ defined in \eqref{equation ladder constructin SEP alpha} is Markov; furthermore,  it is simple to check that, by uniqueness in law of the solution to the martingale problem associated to	 $(L^{\boldsymbol \alpha},\eta)$ (see, e.g., \cite[Chapter 1]{liggett_interacting_2005-1}), its  infinitesimal generator is $L^{\boldsymbol \alpha}$ (we refer to \cite[Section 4.1]{giardina_duality_2009} for further details).

As for the second part of the claim, by definition of the process $\{\eta_t,\,  t \geq 0 \}$ in terms of the process $\{\tilde \eta_t,\, t \geq 0 \}$ and formula \eqref{mild solution ladder}, we obtain, $\mathsf P$-a.s., for all $x \in \Z^d$ and $t \geq 0$, the following expression for $\eta_t(x)$:
	\begin{align}\nonumber
	\eta_t(x) \coloneqq&\ \sum_{i=1}^{\alpha_x} \tilde \eta_t(x,i)\\ =&\ \sum_{i=1}^{\alpha_x} \sum_{y \in \Z^d} \sum_{j=1}^{\alpha_y} \left( \tilde p_t^{\boldsymbol \alpha}((x,i),(y,j))\, \tilde \eta_0(y,i) +	 \int_0^t  \tilde p_{t-s}^{\boldsymbol \alpha}((x,i),(y,j))\, \dd\tilde M^{\boldsymbol \alpha}_s(y,j) \right)\ .
	\end{align}
	Since the infinite summations above are absolutely convergent, we may re-order them so to obtain:
	\begin{equation*}
	\eta_t(x)\ =\ \sum_{y \in \Z^d} \mathcal Y_t(y)\ ,
	\end{equation*}
	where
	\begin{equation}\label{eq:Y_t(y)}
	\mathcal Y_t(y) \coloneqq \sum_{i=1}^{\alpha_x}  \sum_{j=1}^{\alpha_y} \tilde p_t^{\boldsymbol \alpha}((x,i),(y,j))\, \tilde \eta_0(y,j)+\, \int_0^t \sum_{i=1}^{\alpha_x}  \sum_{j=1}^{\alpha_y} \tilde p_{t-s}^{\boldsymbol \alpha}((x,i),(y,j))\, \dd \tilde M^{\boldsymbol \alpha}_s(y,j)\ .
	\end{equation}	
	We observe that, for all sites $x, y \in \Z^d$ and labels $i,i' \in \{1,\ldots,\alpha_x\}$, $j,j'\in \{1,\ldots,\alpha_y\}$, $\tilde p_t^{\boldsymbol \alpha}((x,i),(y,j)) = \tilde p_t^{\boldsymbol \alpha}((x,i'),(y,j'))$; in other words, the transition probabilities $\tilde p_t^{\boldsymbol \alpha}(\cdot,\cdot)$ do not depend on the labels, but only on the sites. Therefore, we  define $\tilde p_t^{\boldsymbol \alpha}(x,y):= \tilde p_t^{\boldsymbol \alpha}((x,i),(y,j)) $. If we combine this with the definition of $\eta_0(y): = \sum_{j =1}^{\alpha_y} \tilde \eta_0(y,j)$,  the expression in \eqref{eq:Y_t(y)} rewrites as follows:
	\begin{equation*}
	\mathcal Y_t(y) = \alpha_x\,  \tilde p_t^{\boldsymbol \alpha}(x,y)\, \eta_0(y)\, +\, \int_0^t \alpha_x\, \tilde p_{t-s}^{\boldsymbol \alpha}(x,y)\, \sum_{j=1}^{\alpha_y} \dd\tilde M^{\boldsymbol \alpha}_s((y,j),\tilde \eta_{s^-})\ .
	\end{equation*}
	Recalling from Section \ref{section strategy proof} the definition of  transition probabilities $\{p^{\boldsymbol \alpha}_t(\cdot,\cdot),\, t \geq 0\}$ associated to $\RWa$ and after observing that 
	\begin{equation}
		p_t^{\boldsymbol \alpha}(x,y)=\sum_{j=1}^{\alpha_y}\tilde p_t^{\boldsymbol \alpha}((x,i),(y,j)) = \alpha_y\,  \tilde p_t^{\boldsymbol \alpha}(x,y)\ ,
	\end{equation} the proof of the identity \eqref{mild solution exclusion} is concluded.

In order to recover the predictable quadratic covariations \eqref{eq:predictable_cov1}--\eqref{eq:predictable_cov2} for  the martingales $\{M^{\boldsymbol\alpha}_t(\cdot),\, t \geq 0\}$, it suffices to combine \eqref{mart} with \eqref{eq:predictable_cov_tilde}--\eqref{eq:predictable_cov_tilde2} and \eqref{equation ladder constructin SEP alpha}; we leave the details to the reader.
\end{proof}

We take the construction and \eqref{equation ladder constructin SEP alpha} in Lemma \ref{lemma mild solution exclusion} as a definition of our partial exclusion process $\SEPa$. In particular, we consider the process $\{\eta_t,\, t\geq 0 \}$  as a Markov functional of the ladder process $\{\tilde \eta_t,\,	 t \geq 0 \}$, whose evolution, in turn, is prescribed in Lemma \ref{lemma mild solution ladder} in terms of the compensated Poisson processes $\{\bar{\mathcal N}(\cdot,\cdot) \}$ in \eqref{poisson} and its initial configuration $\tilde \eta_0 \in \tilde \Xa$. 

However, to any given $\SEPa$-configuration $\eta \in \Xa$ there may correspond, in general, many \textquotedblleft compatible ladder configurations\textquotedblright, namely configurations $\tilde \eta \in \tilde \Xa$ of the following type:
\begin{align*}
\left\{\tilde \eta \in \tilde \Xa: \sum_{i=1}^{\alpha_x} \tilde \eta(x,i) = \eta(x)\ \text{for all}\ x \in \Z^d \right\}\ .
\end{align*} 
Therefore, when we say that the particle system $\{\eta_t,\, t \geq 0 \}$ starts from the configuration $\eta \in \Xa$, we first need to specify how to initialize the underlying ladder process and, then, unequivocally follow the Poissonian source of randomness yielding  \eqref{mild solution exclusion} and \eqref{mart}. We will always  assume that, given an initial configuration $\eta \in \Xa$, the compatible ladder configurations $\tilde \eta \in \tilde \Xa$ are chosen according to some probability distribution \emph{independent} of the compensated Poisson processes in \eqref{poisson}. We can, for instance, make the deterministic choice of filling up the ladders at each site starting from bottom to top.

\section{Proofs of auxiliary results}\label{appendix:ergodic}
In order to fix notation, for all compact subsets $\mathcal K\subseteq \R^d$, $\mathcal C_b(\mathcal K)$ (resp.\ $\mathcal C_c(\mathcal K)$) denotes the space of continuous and bounded (resp.\ compactly supported) functions from $\mathcal K$ to $\R$ endowed with the supremum norm, while $\mathcal M_+(\mathcal K)$ denotes the space of non-negative finite Borel measures on $\mathcal K$ endowed with the weak$^*$ topology w.r.t.\ $\mathcal C_b(\mathcal K)$. Moreover, for all $\mu \in \mathcal M_+(\mathcal K)$ and $F \in \mathcal C_b(\mathcal K)$, we define
\begin{equation}
\mu(F) \coloneqq \int_{\mathcal K} F(u)\, \mu(\dd u)\ .
\end{equation}

\subsection{Proof of Lemma \ref{lemma:ergodic_theorem}}
\begin{proof}[Proof of Lemma \ref{lemma:ergodic_theorem}]
	The methodology of the proof is inspired by \cite[Theorem 8.2.18]{bogachev_measure_2007}.
	
	By applying \cite[Proposition 3.2]{faggionato_stochastic_2020} to the integrable function $g:\boldsymbol \alpha \in \{1,\ldots, \mathfrak c\}^{\Z^d}\mapsto\alpha_0 \in \R$, there exists a (translation invariant) measurable subset $\mathfrak A \subseteq \{1,\ldots, \mathfrak c\}^{\Z^d}$ such that  $\mathcal P(\mathfrak A)=1$ holds,  as well as  
	\begin{equation}\label{eq:pointwise_ergodic1}
		\left|\frac{1}{N^d} \sum_{x \in \Z^d} G(\tfrac{x}{N})\, \alpha_x - \E_{\mathcal P}\left[\alpha_0 \right]\int_{\R^d} G(u)\, \dd u \right|\underset{N\to \infty}\longrightarrow 0
	\end{equation}
	hold for all $\boldsymbol \alpha \in \mathfrak A$ and $G \in \mathcal C_c(\R^d)$, the subspace of $\mathcal C_0(\R^d)$ of compactly supported functions.
	
	In the remainder of this proof, 	$\boldsymbol \alpha \in \mathfrak A$; moreover,  we define
	\begin{equation}
		\mathsf Y^{N,\boldsymbol \alpha}\coloneqq \frac{1}{N^d}\sum_{x\in \Z^d}\delta_{\frac{x}{N}}\, \alpha_x\qquad \text{and}\qquad \mathsf Y\coloneqq \E_{\mathcal P}\left[\alpha_0\right] \dd u
	\end{equation}
as elements in $\mathscr S'(\R^d)$.

	Recall from the proof of Theorem \ref{theorem:uniform_convergence_semigroups} the definitions of the open and closed Euclidean balls $\mathcal B_\ell(u)$ and $\closure{\mathcal B_\ell(u)}$. Then,	for all $\ell > 0$, since  the restriction  map $|_{\closure{\mathcal B_\ell(0)}}: \mathcal C_c(\R^d)\to \mathcal C_c(\closure{\mathcal B_\ell(0)})$ is onto and since $\mathcal C_c(\closure{\mathcal B_\ell(0)})\equiv \mathcal C_b(\closure{\mathcal B_\ell(0)})$, \eqref{eq:pointwise_ergodic1} implies that, for all $\boldsymbol \alpha \in \mathfrak A$,  $\mathsf Y^{N,\boldsymbol \alpha}_\ell$ weakly converge as non-negative finite Borel measures as $N\to \infty$  to $\mathsf Y_\ell$, where
	\begin{equation}
		\mathsf Y^{N,\boldsymbol \alpha}_\ell(\dd u)\coloneqq \frac{1}{N^d}\sum_{\frac{x}{N}\in \closure{\mathcal B_\ell(0)}} \delta_{\frac{x}{N}}(\dd u)\, \alpha_x\qquad \text{and}\qquad \mathsf Y_\ell(\dd u)\coloneqq \E_{\mathcal P}\left[\alpha_0 \right]		 \ind_{\{u \in \closure{\mathcal B_\ell(0)}\}}\, \dd u\ .
	\end{equation}
	By the compactness of $\closure{\mathcal B_\ell(0)}\subseteq \R^d$, for all $\delta > 0$, there exists  a finite sub-cover $\mathcal U_\ell(\delta)\coloneqq\left\{\mathcal B_\delta(u_i) \right\}_{i=1}^n$ of open balls of radius $\delta > 0$ (with 	$n=n(\delta)\in \N$). Moreover, by defining recursively $V_1\coloneqq \mathcal B_\delta(u_1)\cap \closure{\mathcal B_\ell(0)}$ and  $V_i\coloneqq  \{\mathcal B_\delta(u_i)\cap\closure{\mathcal B_\ell(0)}\}\setminus V_{i-1}$, it is simple to check that the pairwise disjoint sets $\mathcal V_\ell(\delta)\coloneqq \left\{V_i \right\}_{i=1}^n$ 	cover $\closure{\mathcal B_\ell(0)}$ and $\mathsf Y_\ell(\partial V_i)=0$ for all $i=1,\ldots, n$, where $\partial V_i$ denotes the boundary of $V_i$ in the subspace topology on $\closure{\mathcal B_\ell(0)}$. Hence, 
	\begin{align}\nonumber\label{eq:upper_bound_ergodic}
		\sup_{F\in \mathcal F}\left|\mathsf Y^{N,\boldsymbol \alpha}_\ell(F)-\mathsf Y_\ell(F) \right|\leq&\ \sup_{F \in \mathcal F}\sum_{i=1}^n \frac{1}{N^d}\sum_{\frac{x}{N}\in V_i}\left|F(\tfrac{x}{N})-F(u_i) \right|\alpha_x  \\
		\nonumber
		+&\ \sup_{F\in \mathcal F} \sup_{u\in \R^d}\left|F(u)\right| \sum_{i=1}^n \left|\frac{1}{N^d}\sum_{\frac{x}{N}\in V_i}\alpha_x - \int_{V_i}\E_{\mathcal P}\left[\alpha_0 \right] \dd u  \right|	\\
		+&\ \sup_{F \in \mathcal F}\sum_{i=1}^n\E_{\mathcal P}\left[\alpha_0 \right] \int_{V_i}\left|F(u)-F(u_i) \right|\, \dd u\ .
	\end{align}
	The boundedness of $\mathcal F$
	(see \eqref{eq:bounded}),
	and the convergence (recall that $\mathsf Y^{N,\boldsymbol \alpha}_\ell$ weakly converges to $\mathsf Y_\ell$ as $N\to \infty$ as well as	 $\mathsf Y_\ell(\partial V_i)=0$ for all $i=1,\ldots, n$)
	\begin{equation}
		\mathsf Y^{N,\boldsymbol \alpha}_\ell(V_i)\underset{N\to \infty}\longrightarrow \mathsf Y_\ell(V_i)\ ,\qquad i = 1, \ldots, n\ ,
	\end{equation}
	ensure that, for all $\delta >0$,  
	the second term on the r.h.s.\ in \eqref{eq:upper_bound_ergodic} vanishes as $N\to \infty$:
	\begin{equation}\label{eq:second_term}
		\sup_{F\in \mathcal F} \sup_{u\in \R^d}\left|F(u)\right| \sum_{i=1}^n \left|\frac{1}{N^d}\sum_{\frac{x}{N}\in V_i}\alpha_x - \int_{V_i}\E_{\mathcal P}\left[\alpha_0 \right] \dd u  \right|\underset{N\to \infty}\longrightarrow 0\ .	
	\end{equation}	
	The first and third terms on the r.h.s.\ in \eqref{eq:upper_bound_ergodic} are both bounded above by
	\begin{equation}
		\sup_{\substack{u,v \in \R^d\\|u-v|<\delta}}\sup_{F \in \mathcal F}\left|F(u)-F(v) \right|\left\{\mathsf Y^{N,\boldsymbol \alpha}_\ell(\closure{\mathcal B_\ell(0)})+ \mathsf Y_\ell(\closure{\mathcal B_\ell(0)}) \right\}\ ;
	\end{equation}
	hence, by the definition \eqref{eq:equicontinuous} of 	equicontinuity of the subset $\mathcal F\subseteq \mathcal C_0(\R^d)$
	and
	\begin{equation}
		\limsup_{N \to \infty}\mathsf Y^{N,\boldsymbol \alpha}_\ell(\closure{\mathcal B_\ell(0)})+ \mathsf Y_\ell(\closure{\mathcal B_\ell(0)})= 2 \mathsf Y_\ell(\closure{\mathcal B_\ell(0)})<\infty\ ,
	\end{equation}
	we obtain
	\begin{equation}\label{eq:first_third_terms}
		\lim_{\delta \downarrow 0}\limsup_{N\to \infty}	\sup_{F \in \mathcal F}\left\{\sum_{i=1}^n \frac{1}{N^d}\sum_{\frac{x}{N}\in V_i}\left|F(\tfrac{x}{N})-F(u_i) \right|\alpha_x  +\sum_{i=1}^n\E_{\mathcal P}\left[\alpha_0 \right] \int_{V_i}\left|F(u)-F(u_i) \right|\, \dd u \right\}=0\ .
	\end{equation}
	Hence, \eqref{eq:second_term} and \eqref{eq:first_third_terms} combined with \eqref{eq:upper_bound_ergodic} yield, for all $\ell > 0$,
	\begin{equation}\label{eq:convergence_inside}
		\sup_{F\in \mathcal F}\left|\mathsf Y^{N,\boldsymbol \alpha}_\ell(F)-\mathsf Y_\ell(F)  \right|\underset{N\to \infty}\longrightarrow 0\ .
	\end{equation}

	The uniform integrability assumption (see \eqref{eq:uniform_integrability}) 
	and the upper bound $\alpha_x\leq \mathfrak c< \infty$ (see Assumption \ref{definition  disorder}) ensure	 	
	\begin{equation}\label{eq:convergence_outside}
		\lim_{\ell \to \infty}\limsup_{N\to \infty}	\sup_{F \in \mathcal F}\left\{\frac{1}{N^d}\sum_{\left|\frac{x}{N}\right|>\ell} \left|F(\tfrac{x}{N}) \right|\alpha_x + \E_{\mathcal P}\left[\alpha_0\right] \int_{\left\{|u|>\ell\right\}}\left|F(u)\right|\dd u\right\}=0\ .	
	\end{equation}
	The triangle inequality
	\begin{align}\nonumber
		\sup_{F\in \mathcal F}\left|\mathsf Y^{N,\boldsymbol \alpha}(F)-\mathsf Y(F) \right|\leq&\ \sup_{F \in \mathcal F}\left|\mathsf Y^{N,\boldsymbol \alpha}_\ell(F)-\mathsf Y_\ell(F) \right|\\
		+&\ \sup_{F \in \mathcal F}\left\{\frac{1}{N^d}\sum_{\left|\frac{x}{N}\right|>\ell} \left|F(\tfrac{x}{N}) \right|\alpha_x + \E_{\mathcal P}\left[\alpha_0\right] \int_{\left\{|u|>\ell\right\}}\left|F(u)\right|\dd u\right\}\ ,
	\end{align}	
	which holds for all $\ell> 0$ and $N\in \N$, combined with \eqref{eq:convergence_inside} and \eqref{eq:convergence_outside}, yields the desired result. 
\end{proof} 

\subsection{Proof of Corollary \ref{corollary:l1}}

\begin{proof}[Proof of Corollary \ref{corollary:l1}]
In what follows, let $\boldsymbol \alpha$ be an environment in the subset $\mathfrak A\cap \mathfrak B\subseteq \{1,\ldots, \mathfrak c\}^{\Z^d}$ (see \eqref{eq:A} and \eqref{eq:B}).   Fix $T > 0$ and $G \in \mathscr S(\R^d)\subseteq \mathcal C_0(\R^d)$. Let $G^+$ and $G^-$ be the positive  and negative parts of $G$ ($G=G^+-G^-$); then, $G^\pm \in    L^1(\R^d)\cap \mathcal C_0(\R^d)$  (hence they satisfy 
	\eqref{eq:uniform_convergence_semigroups})  and 
	 there exist functions $H^\pm \in \mathscr S(\R^d)$ (see, e.g., \cite[Proposition 5.3]{redig_symmetric_2020} for an explicit construction) such that
	\begin{align}\label{eq:boundGH}
		0 \leq G^\pm(u) \leq H^\pm(u)\ ,\qquad u \in \R^d\ .
	\end{align} 
 As a consequence, there exist constants $C^\pm > 0$ such that
	\begin{align}\label{eq:integrability}
		\sup_{0 \leq t \leq T} |\mathcal S^\varSigma_t G^\pm(u)| \leq \frac{C^\pm}{1+|u|^{2d}}\ ,\qquad u \in \R^d\ .
	\end{align}
	This follows from the bounds \eqref{eq:boundGH}, the fact that	 $\mathcal S^\varSigma_t$ acts as convolution with a non-degenerate Gaussian kernel and the use of Fourier transformation in $\mathscr S(\R^d)$. Moreover,  because of the uniform continuity of $G^\pm$ and the contractivity of the semigroup in $\mathcal C_0(\R^d)$, we have
	\begin{align*}
		\sup_{t\in [0,T]} \sup_{|u-v|<\delta}\left|\mathcal S^\varSigma_tG^\pm(u)-\mathcal S^\varSigma_tG^\pm(v)\right|\ \leq\  \sup_{t\in [0,T]} \sup_{|u-v|<\delta}\left|G^\pm(u)-	G^\pm(v)\right|\  \underset{\delta \to 0}\longrightarrow\ 0\ .
	\end{align*}
As a consequence, for all  $T > 0$,  both subsets of $\mathcal C_0(\R^d)$ given by
\begin{equation}
	\mathcal F_{[0,T]}(G^\pm)\coloneqq\left\{\mathcal S^\varSigma_t G^\pm \in \mathcal C_0(\R^d): t \in [0,T] \right\}
\end{equation}
satisfy the assumptions in  Lemma \ref{lemma:ergodic_theorem}. Therefore,  since $\boldsymbol \alpha \in \mathfrak A$,  Lemma \ref{lemma:ergodic_theorem} ensures that, for all $G\in \mathscr S(\R^d)$ and $T >  0$,  we have
	\begin{align}\label{eq:uniform_convergence_integrals}
		\sup_{t\in[0,T]} \left|\frac{1}{N^d} \sum_{x \in \Z^d} \mathcal S^\varSigma_t G^\pm(\tfrac{x}{N})\, \alpha_x - \int_{\R^d} \mathcal S^\varSigma_t G^\pm(u)\, \E_{\mathcal P}\left[\alpha_0\right] \dd u \right|\ \underset{N\to \infty}\longrightarrow\ 0\ .
	\end{align}

	Let us now prove
	\begin{align} \label{l1 convergence propagators pm}
		\sup_{t\in [0,T]} \frac{1}{N^d} \sum_{x \in \Z^d} \left|S^{N,\boldsymbol \alpha}_{tN^2}G^\pm(\tfrac{x}{N}) - \mathcal S^\varSigma_t G^\pm(\tfrac{x}{N}) \right|\alpha_x\ \underset{N \to \infty}\longrightarrow\  0\ ,
	\end{align}
	from which \eqref{eq:l1sup_convergence} follows.

	Since $|c|= c +2 \max\{-c,0 \}$ for all $c \in \R$, we have
	\begin{align}\label{eq:bound222}\nonumber
		&\sup_{t\in [0,T]} \frac{1}{N^d} \sum_{x \in \Z^d} \left|S^{N,\boldsymbol \alpha}_{tN^2}G^\pm(\tfrac{x}{N}) - \mathcal S^\varSigma_t G^\pm(\tfrac{x}{N}) \right|\alpha_x\\
		\nonumber
		&\leq\ \sup_{t\in[0,T]} \frac{1}{N^d} \sum_{x \in \Z^d} \left(S^{N,\boldsymbol \alpha}_{tN^2}G^\pm(\tfrac{x}{N}) - \mathcal S^\varSigma_t G^\pm(\tfrac{x}{N})\right)\alpha_x \\
		&+\ \sup_{t\in [0,T]} \frac{2}{N^d} \sum_{x \in \Z^d}  \max\left\{\mathcal S^\varSigma_t	 G^\pm(\tfrac{x}{N})-S^{N,\boldsymbol \alpha}_{tN^2}G^\pm(\tfrac{x}{N}), 0 \right\}\alpha_x\ .
	\end{align}
As for  the first term in the r.h.s.\ above, by  detailed balance (see \eqref{eq:detailed_balance_probabilities}), 
	$\sum_{x \in \Z^d}  p^{\boldsymbol \alpha}_{tN^2}(y,x)= 1$, 
	as well as  $\int_{\R^d} \mathcal S^\varSigma_t G^\pm(u)\, \dd u = \int_{\R^d} G^\pm(u)\, \dd u $, we obtain 
	\begin{align*}
		& \sup_{t\in [0,T]} \left|\frac{1}{N^d} \sum_{x \in \Z^d} \left(S^{N,\boldsymbol \alpha}_{tN^2}G^\pm(\tfrac{x}{N}) - \mathcal S^\varSigma_t G^\pm(\tfrac{x}{N})\right)\alpha_x\right|\\ 
		&=\  \sup_{t\in [0,T]}\left|\frac{1}{N^d}  \sum_{y \in \Z^d} G^\pm(\tfrac{y}{N})\, \alpha_y \sum_{x \in \Z^d}   p^{\boldsymbol \alpha}_{tN^2}(y,x) - \frac{1}{N^d} \sum_{x \in \Z^d} \mathcal S^\varSigma_{t-s} G^\pm(\tfrac{x}{N})\, \alpha_x \right|\\
		&\leq\ \left|\frac{1}{N^d} \sum_{x \in \Z^d} G^\pm(\tfrac{x}{N})\, \alpha_x - \int_{\R^d} G^\pm(u)\,  \E_{\mathcal P}\left[\alpha_0\right] \dd u \right|\\
		&+\ \sup_{t \in [0,T]}	\left|\frac{1}{N^d} \sum_{x \in \Z^d} \mathcal S^\varSigma_t G^\pm(\tfrac{x}{N})\, \alpha_x - \int_{\R^d} \mathcal S^\varSigma_t G^\pm(u)\,  \E_{\mathcal P}\left[\alpha_0\right] \dd u \right|\ ;
	\end{align*}
thus, the first expression on the r.h.s.\ of \eqref{eq:bound222} vanishes as $N \to \infty$ by \eqref{eq:uniform_convergence_integrals}.

	Moreover, we have, for all $N \in \N$ and $x \in \Z^d$, 
	\begin{align}\label{eq:kkk}
		& \sup_{t\in [0,T]} \max\left\{\mathcal S^\varSigma_tG^\pm(\tfrac{x}{N})- S^{N,\boldsymbol\alpha}_{tN^2}G^\pm(\tfrac{x}{N}),\, 0 \right\} \alpha_x \leq \sup_{t\in [0,T]} \mathcal S^\varSigma_t G^\pm(\tfrac{x}{N})\, \alpha_x\ .
	\end{align}
	Therefore,  for all $\ell > 0$ and combining \eqref{eq:kkk} and \eqref{eq:integrability}, we obtain		
	\begin{align}\nonumber
		&\limsup_{N \to \infty} \sup_{t\in [0,T]} \frac{2}{N^d} \sum_{x \in \Z^d}  \max\left\{\mathcal S^\varSigma_t G^\pm(\tfrac{x}{N})-S^{N,\boldsymbol \alpha}_{tN^2}G^\pm(\tfrac{x}{N}),\, 0 \right\}\alpha_x\\
		\label{eq:bound4}
		&\leq   \limsup_{N \to \infty}     \sup_{t\in [0,T]} \sup_{
			|\frac{x}{N}| \leq \ell}\left|\mathcal S^\varSigma_t G^\pm(\tfrac{x}{N})-S^{N,\boldsymbol \alpha}_{tN^2}G^\pm(\tfrac{x}{N})\right| \frac{1}{N^d}\sum_{\left| \frac{x}{N}\right|\leq \ell}\alpha_x\\
		\label{eq:bound5}
		&+ \limsup_{N \to \infty}   \frac{2}{N^d} \sum_{\left|\frac{x}{N}\right| > \ell }  \frac{C^\pm\, \alpha_x}{1+|\tfrac{x}{N}|^{2d}}\ .
	\end{align}
By Theorem \ref{theorem:uniform_convergence_semigroups} applied to the functions $G^\pm$ and $\sup_{N\in \N} \frac{1}{N^d}\sum_{\left|\frac{x}{N}\right|\leq \ell}\alpha_x<\infty$,  \eqref{eq:bound4} equals zero for all $\ell > 0$, while \eqref{eq:bound5} vanishes as $\ell \to \infty$.
This concludes the proof.
\end{proof}

\end{appendices}

\paragraph{Acknowledgements.}
The authors would like to thank Marek Biskup and Alberto Chiarini for useful suggestions and Cristian Giardin\`{a}, Frank den Hollander and Shubhamoy Nandan for  inspiring discussions. S.F. acknowledges Simona Villa for  her help in creating the  picture.
Furthermore, the authors thank two anonymous referees for the careful reading of the manuscript. 
S.F. acknowledges  financial support from NWO via the  grant TOP1.17.019. F.S.\ acknowledges financial support from NWO via the TOP1 grant 613.001.552 as well as  funding from the European Union's Horizon 2020 research and innovation programme under the Marie-Sk\l{}odowska-Curie grant agreement  No.\ 754411.

\bibliographystyle{acm}

\end{document}